\documentclass[reqno, a4paper]{amsart}
\usepackage{amsmath, amssymb, amsthm, enumitem}
\usepackage{mathtools}
\usepackage{placeins}
\usepackage{subcaption,float}
\usepackage{tcolorbox}
\usepackage[normalem]{ulem}

\pdfoutput=1
\usepackage{color}
\def\comment#1{}

\usepackage{hyperref}

\usepackage{txfonts}

\usepackage{tikz}
\usetikzlibrary{external}
\usetikzlibrary{backgrounds,patterns,decorations.pathreplacing,decorations.pathmorphing}
\usepackage{caption}
\usepackage{accents}
\usepackage{mathrsfs}

\newcommand{\SW}{\mathcal{SW}_p}

\newtheorem{theorem}{Theorem}

\newtheorem{assumptions}[theorem]{Assumptions}

\newtheorem{definition}[theorem]{Definition}
\newtheorem{lemma}[theorem]{Lemma}
\newtheorem{proposition}[theorem]{Proposition}

{Important Convention}

\theoremstyle{remark}
\newtheorem{remark}[theorem]{Remark}

\newtheorem{example}[theorem]{Example}

 \newcommand{\eps}{\varepsilon}

 \newcommand{\PP}{\mathbb{P}}
 \newcommand{\W}{\mathcal{W}}

 \renewcommand{\phi}{\varphi}
\newcommand{\pp}{\mathcal P}

\newcommand{\E}{\mathbb{E}}

\renewcommand{\PP}{\mathbb{P}}

\newcommand{\Q}{\mathbb{Q}}

\newcommand{\indep}{\rotatebox[origin=c]{90}{$\models$}}

\usepackage{subcaption}
\usepackage{graphicx}

\newcommand{\bes}{\begin{subequations}}
\newcommand{\ees}{\end{subequations}}
\newcommand{\eea}{\end{eqnarray}}

\usepackage{bbm}

\renewcommand{\eps}{\varepsilon}
\renewcommand{\epsilon}{\varepsilon}

\newcommand{\fourIdx}[5]{%
\setbox1=\hbox{\ensuremath{^{#1}}}%
 \setbox2=\hbox{\ensuremath{_{#2}}}%
 \setbox5=\hbox{\ensuremath{#5}}%
 \hspace{\ifnum\wd1>\wd2\wd1\else\wd2\fi}%
 \ensuremath{\copy5^{\hspace{-\wd1}\hspace{-\wd5}#1\hspace{\wd5}#3}%
 _{\hspace{-\wd2}\hspace{-\wd5}#2\hspace{\wd5}#4}%
 }}

\numberwithin{equation}{section}
\numberwithin{theorem}{section}

\renewcommand{\subset}{\subseteq}

\newcommand{\mo}{\bar{M}}
\newcommand{\vo}{\bar{v}}
\newcommand{\so}{\bar{\sigma}}

\renewcommand{\mathrm}{}

\newcommand{\mylabel}[2]{#2\def\@currentlabel{#2}\label{#1}}

\begin{document}

\title[Specific Wasserstein divergence between continuous martingales]{Specific Wasserstein divergence between continuous martingales}

\author{Julio Backhoff-Veraguas, Xin Zhang}

\begin{abstract}

Defining a divergence between the laws of continuous martingales is a delicate task, owing to the fact that these laws tend to be singular to each other. An important idea, by Gantert \cite{gantert1991einige}, is to instead consider a scaling limit of the relative entropy between such continuous martingales sampled over a finite time grid. This gives rise to the concept of specific relative entropy. In order to develop a general theory of divergences between continuous martingales, it is  natural to replace the role of the relative entropy by a different notion of discrepancy between finite dimensional probability distributions. We take a first step in this direction, taking a power $p$ of the Wasserstein distance. We call the newly obtained scaling limit the specific $p$-Wasserstein divergence.

In this paper we prove that the specific $p$-Wasserstein divergence is well-defined, exhibit an explicit expression for it, and compare it with the specific relative entropy and adapted Wasserstein distance on a class of SDEs. Then we consider specific $p$-Wasserstein divergence optimization over the set of win-martingales. Finally, we characterize the solution of such optimization problems for all $p>0$ and, surprisingly, we single out the case $p=1/2$ as the one with the best probabilistic properties.

\end{abstract}

\thanks{This research was funded 
in whole or in part by the Austrian Science Fund (FWF) DOI 10.55776/P36835. 
}

\maketitle

\section{Introduction}

The aim of this paper is to introduce a novel notion of divergence between continuous martingales, and thereafter to fully study and solve divergence optimization problems over the set of win martingales (used as models for prediction markets \cite{MR3096465}). 

Identifying two real-valued continuous martingales as probability measures $\mathbb Q, \mathbb P$ on the continuous path space \[C([0,1];\mathbb{R}),\] a natural choice for divergence would be the relative entropy \[H(\mathbb Q || \mathbb P)=\int \frac{d\mathbb Q}{d \mathbb P}\, \log\big(\frac{d\mathbb Q}{d \mathbb P}\big) \, d\mathbb P,\] or its generalization, the $f$-divergence $\int f (\frac{d\mathbb Q}{d \mathbb P})\, d\mathbb P$. This naive approach leads to an immediate difficulty, namely, that the laws of continuous martingales tend to be singular to each other and hence have a trivial divergence in the above sense. As an example, the reader can consider  $\mathbb Q, \mathbb P$ to be respectively the laws of Brownian motion and two times the Brownian motion, and check that these measures are concentrated on disjoint sets. 

A natural approach to circumvent this difficulty is to rather discretize the martingales in time, compare them at the discrete-time level, and then consider a (scaling) limit in which the time mesh-size goes to zero. This approach was introduced by Gantert in \cite{gantert1991einige}, and more recently refined by F\"ollmer in \cite{MR4433813}, leading to the notion of specific relative entropy between continuous martingales. Independently, this concept appeared in mathematical finance in the works of Avellaneda et al \cite{Av01} and Cohen and Dolinsky \cite{CoDo22}, {but see also Dolinsky and Zhang \cite{DoZh25} as well as the recent articles by Benamou et al \cite{BeChHoLoVi24,BeChLo24} for contemporary applications.} The key in this construction is to use the conventional relative entropy when comparing the laws of the time-discretized martingales. In this article we introduce and study the \emph{specific Wasserstein divergence}, obtained by considering a power of the Wasserstein distance instead of the relative entropy in the aforementioned scaling limit. We view our work as a first step towards a general theory of divergences between continuous martingales. 

{ At this point the reader may be wondering about the need for a theory of divergences between continuous martingales. Why does one need more objects than the specific relative entropy? Although our reasons are explained in detail throughout the article, we collect here our three main motivations:
\begin{itemize}
    \item The specific relative entropy is a tricky mathematical object: in the pre-limit its value can easily be equal to infinity, and this already happens in natural examples. Furthermore, to obtain a tractable formula in the continuous-time scaling limit, one needs very strong assumptions on the martingales involved. The specific Wasserstein divergence, proposed in this article, has superior properties when it comes to finiteness and convergence.
    \item We want to understand the reasons why the construction of the specific relative entropy works. Our study indicates that there are only two properties of the relative entropy which are crucial when one wants to pass to a scaling limit. First, its so-called additive property (also called additive decomposition or chain-rule), namely that the multivariate relative entropy is equal to the an expected sum of univariate relative entropies applied to the conditional laws of the measures involved. Second, the fact that if two univariate measures are (almost) Gaussian, then their relative entropy is (almost) explicit, and very tractable, in terms of the involved variances. Following this rationale, one can define a multivariate divergence via the expected sum of univariate divergences applied to the conditional laws. Specializing to the case when the univariate divergence is a power of the Wasserstein distance, one obtains explicit and tractable expressions in the (almost) Gaussian case. Passing to the scaling-limit we obtain our specific Wasserstein divergences.
    \item We want to give the modeler more choices, with a sound discrete-to-continuous time limit theory, when it comes to applications e.g.\ in control and finance. As it happens, the specific relative entropy of a martingale with diffusion coefficient $(\sigma_t)_t$, with respect to Brownian motion as a reference process, has in the best case scenario the formula $\frac{1}{2}\E \left[\int_0^1 \left\{ \sigma_t^2-1-\log \sigma_t^2 \right\} \, dt \right]$, while for the specific Wasserstein divergence we will eventually obtain in one of our main results the formula $(2/\pi)^{p/2}\mathbb E_{\mathbb Q}\left [\int_0^1 \left|\,|\sigma_t|-1\,\right |^{p} \, dt \right ]$, where $p>0$ is a parameter of our choice. This provides a family of objects that can be used as loss functions, or regularizing costs, in calibration problems. Alternatively, they could be used to describe neighborhoods around a reference model (here Brownian motion) in robust optimization problems. As an advantage vis-\`a-vis transport based distances, there is no need for further optimization over a set of couplings (and this remains true if we replace Brownian motion by other reference measure such as nice diffusion martingales). As an advantage  vis-\`a-vis the specific relative entropy, the right choice of $p$ can either induce or discourage $\sigma$ to get close to zero, depending on what the modeler wants. Furthermore, the expression of  the specific Wasserstein divergence is more algebraic, which we believe can be advantageous in certain situations, e.g.\ the Hamiltonian in stochastic control problems, or even the marginal distributions of an optimal martingale, can be made more tractable depending on $p$.
\end{itemize}
}

\subsection{Specific Wasserstein divergence}

Denote throughout by $\mathcal P(\mathcal X)$ the space of probability measures on a Polish space $\mathcal X$. 

Given a function $F: \mathcal{P}(\mathbb{R}) \times \mathcal{P}(\mathbb{R}) \to \mathbb{R}_+$ which only vanishes on the diagonal, we may interpret  $D_F^1(\mu,\nu):=F(\mu,\nu)$ as the \emph{discrepancy} between $\mu$ and $\nu\in \mathcal{P}(\mathbb{R})$. Inductively, after having defined $D_F^{N-1}: \mathcal P(\mathbb R^{N-1})\times \mathcal P(\mathbb R^{N-1}) \to  \mathbb{R}_{+}$, a discrepancy functional between elements in $\mathcal P(\mathbb R^{N-1})$, then a natural choice for a discrepancy functional between elements in $\mathcal P(\mathbb R^{N})$ is given by 
\begin{align}\label{eq:intro_discrep}
D_F^{N}(\mathbb Q || \mathbb P):=D_F^{N-1}(\mathbb Q_{1:N-1} || \mathbb P_{1:N-1})+\int F(\mathbb Q_N^{x_{1:N-1}}, \mathbb P_N^{x_{1:N-1}}) \, d\mathbb Q_{1:N-1}(x_{1:N-1}),
\end{align}
where $\mathbb{Q}_{1:N-1}$ is the projection of $\mathbb Q$ on the first $N-1$ coordinates, $x_{1:N-1}=(x_1,\dotso, x_{N-1}) \in \mathbb{R}^{N-1}$,  $\mathbb Q_N^{x_{1:N-1}}$ is the conditional law of the $N$-th marginal given the previous trajectory $x_{1:N-1}$, with similar notations for $\mathbb{P}_{1:N-1}$ and $\mathbb P_N^{x_{1:N-1}}$. 

If now  $\mathbb Q ,\mathbb P \in \mathcal{P}(C([0,1];\mathbb{R}))$, a natural choice for a discrepancy functional between  $\mathbb Q $ and $\mathbb P $ is to take
\begin{align}\label{eq:introlimit}
\lim\limits_{N \to \infty} c^F_N D^N_F((\mathbb Q)_N||(\mathbb P)_N) ,
\end{align}
assuming that the limit exists and that a universal scaling sequence $\{c^F_N\}_N\subset \mathbb R_+$ has been found. Here and throughout we use the notation $(\mathbb Q)_N$ for the push-forward (i.e., image measure) of $\mathbb Q$ under the map
\[C([0,1];\mathbb R)\ni \omega \mapsto (\omega_{1/N},\omega_{2/N},\dots,\omega_{N/N}), \]
and likewise for $(\mathbb P)_N$.

We remark that \eqref{eq:intro_discrep} is akin to the time consistency property in the theory of dynamic risk measures, and is very natural given the temporal structure of processes like martingales.  Therefore the discrepancy functional defined  as in \eqref{eq:introlimit} could be used as penalty functions for risk measures of processes; see Remark~\ref{rmk:timeconsistency} for more details.
 Taking $F$ to be the relative entropy $H$, the construction described so far gives rise (with $c^F_N=N^{-1}$) to the specific relative entropy. Thus the specific relative entropy can be understood as the rate of increase of the relative entropy as the number of marginals taken into account increases.
 
 In this paper, we will be concerned with the case \[F=\W_1^p,\] where $\W_1$ is the celebrated $1$-Wasserstein distance and $p$ is any positive real number. We call the resulting object the  \emph{specific $p$-Wasserstein divergence} (though sometimes we omit to mention the parameter $p$), and denote it by $\mathcal{SW}_p$. Moreover, we will only be concerned with continuous martingale laws throughout. As a side note, we remark that taking a different Wasserstein distance than $\mathcal W_1$ would change very little the results in this paper. 
 
 We proceed to describe our first main result: Suppose $\mathbb Q$ is the law of a continuous martingale starting wlog.\ at 0 and admitting an absolutely continuous quadratic variation with density denoted by $\sigma^2$. Suppose that $\mathbb P$ is the law of standard Brownian motion. In Theorem~\ref{thm:convergence} we obtain the existence of the specific Wasserstein divergence together with its explicit formula
\begin{align*}
\SW(\mathbb Q || \mathbb P):= & \lim_{\substack{N=2^n \\ n\to\infty}} N^{p/2-1}D_{\W}^{N,p}((\mathbb Q)_N || (\mathbb P)_N ) =(2/\pi)^{p/2}\mathbb E_{\mathbb Q}\left [\int_0^1 \left|\,|\sigma_t|-1\,\right|^{p} \, dt \right ],
\end{align*}
were we denote here and throughout $D^N_{\W^p_1}$ by $D_{\W}^{N,p}$, for brevity.
In fact, in Theorem~\ref{thm:convergence} we can allow the law $\mathbb P$ to be more general than the Wiener measure, e.g.\ it can come from the law of a martingale diffusion with a well-behaved volatility coefficient. One notable aspect of this result is that the assumptions needed are vastly weaker than the ones needed for the corresponding result in the case of the specific relative entropy. In particular, we can handle the situation where $\mathbb P$ is the constant martingale (i.e., a Brownian motion with zero volatility), and more generally, certain cases where $(\mathbb Q)_N$ and $(\mathbb P)_N$ may be singular to each other for every $N$. We consider this result a first step towards building a general theory of divergences between continuous martingales, since unlike the case of the specific relative entropy, the construction here gives a whole family of divergences.

Related to the above main result, we recover in Proposition~\ref{prop:inequalities} a functional inequality by F\"ollmer \cite{MR4433813} concerning (in our terminology) the specific Wasserstein divergence, specific relative entropy, and an adapted Wasserstein distance. Our proof method is based on Theorem~\ref{thm:convergence} and is different from the original approach by this author.  Moreover, we find that the specific Wasserstein divergence and the adapted Wasserstein distance induce the same topology on given classes of SDEs; see Remark~\ref{rmk:toposame}.

\subsection{Optimization over win-martingales}

The main application of the concept of specific Wasserstein divergence that we want to put forward in this paper, concerns a problem of optimization over the set of win-martingales. A continuous-time martingale $(X_t)$ over time $t\in [0,1]$ is called a win-martingale if it starts with a deterministic position $X_0=x_0 \in (0,1)$ and ends up at time 1 on either $0$ or $1$. In other words, it transports $\delta_{x_0}$ at time $0$ to $Bernoulli(x_0)$ at time $1$. Such martingales have been proposed as models of prediction markets (cf.\ \cite{MR3096465}) and optimization problems over the set of win-martingales were proposed by Aldous \cite{aldous_winmartingale}. { In the context of information acquisition, win-martingales are also known as belief-martingales \cite{El17,JAE15,Zh22}.}
 
 Given a fixed initial position $x_0 \in (0,1)$ and $p>0$, our aim is to optimize the specific Wasserstein divergence $\mathcal{SW}_p(\mathbb Q || \mathbb P_\delta)$ among all continuous win-martingales $\mathbb Q$ started at $x_0$ and admitting an absolutely continuous quadratic variation, whereby $\mathbb P_\delta$ denotes the constant martingale {(see Remark \ref{rmk:2.2} for a discussion on how we interpret the transition probabilities of $\mathbb P_\delta$)}.  More specifically, we maximize $\mathcal{SW}_p(\mathbb Q || \mathbb P_\delta)$ for $p \in (0,2)$ and minimize $\mathcal{SW}_p(\mathbb Q || \mathbb P_\delta)$ for $p \in [2,\infty)$. { There are two main motivations for analyzing this problem. On the one hand, the quantity $\mathcal{SW}_p(\mathbb Q || \mathbb P_{\delta})$ can be interpreted as an entertainment utility functional, relevant to the design of mystery novels, political primaries, game shows, and auctions; see, e.g., \cite{JAE15}. On the other hand, in Lemma~\ref{lem:valuefunctionlimit}, we show that $\mathcal{SW}_p(\cdot|| \mathbb P_{\delta})$ arises as the limit (in the sense of $\Gamma$-convergence, for $p>2$) of $\mathcal{SW}_p(\cdot || \mathbb P_{\epsilon})$  as $\epsilon \to 0$, where $\mathbb P_{\epsilon}$ denotes the law of the scaled Wiener process $(\epsilon W_t)_t$. Furthermore, we establish the convergence of the value functions associated with the corresponding optimization problems over win-martingales.
As it turns out, although the win-martingale optimization problem with cost $\mathcal{SW}(\cdot, \mathrm{Law}(\epsilon B))$ can be characterized via HJB equations, we currently do not know how to identify the optimizers. This is because the equations, derived either from the first-order condition in Lemma 3.3 or from the HJB equations, are quite complicated. As a result, it is unclear whether they admit regular solutions, let alone closed-form expressions for the optimal $\sigma$. Fortunately, for the cost $\mathcal{SW}(\cdot, \mathbb{P}_\delta)$, the optimization problem can be (almost) explicitly solved, allowing us to establish connections with the Schr\"{o}dinger problem and the filtering problem.

}

 These optimization problems are related to the one in \cite{BBexciting}, wherein the specific relative entropy with respect to standard Brownian motion is minimized over this set of martingales. Similarly to this reference, we employ first order conditions to characterize in  Proposition~\ref{prop:imsoln}, via ordinary differential equations, the (Markovian) volatility coefficient of a candidate optimal martingale. Then, in Section~\ref{sec:verification}, the optimality of the candidate is verified by making use of the associated HJB equation and stochastic analysis arguments. Hence we obtain semi-explicit solutions for a whole family of continuous-time martingale optimal transport problems (which cannot be transformed into an Skorokhod Embedding Problem), usually considered a difficult task.
 
 We identify {three cases} where the solution to our problem is fully explicit. The first is for $p=1$, where the solution is a so-called Bass martingale (see \cite{backhoff2020martingale,2023arXiv230611019B}). As this object is well studied we do not explore this case in any detail. { The second is when $p=2$. In this case every win-martingale is optimal, due to It\^{o}'s isometry. However we identify a particular one which can be viewed as the limit of optimal martingales as $p \to 2$. As it turns out, this is precisely the so-called Wright-Fisher (martingale) diffusion; see Remark~\ref{rmk:wf-diffusion}.
The third case is $p=1/2$, which we will explore in much more depth.} In this setting the unique optimal win martingale solves the SDE
 \[dM_t=\sqrt{\frac{2}{1-t}}M_t(1-M_t)\, dB_t. \]
 
 In order to give some intuition, we provide some numerical simulations of the Bass martingale and $(M_t)$; see Figure~\ref{fig:test1} and Figure \ref{fig:test2}. The reader may notice that the former tends to explore the space relatively faster than the latter. Indeed Lemma~\ref{lem:cx} below justifies that for each moment of time $t \in [0,1]$, the distribution of the Bass martingale is greater than that of $M_t$ in the sense of convex order.

  \begin{figure}[H]
\centering
\begin{minipage}{.5\textwidth}
  \centering
  \includegraphics[width=4cm,height=3cm]{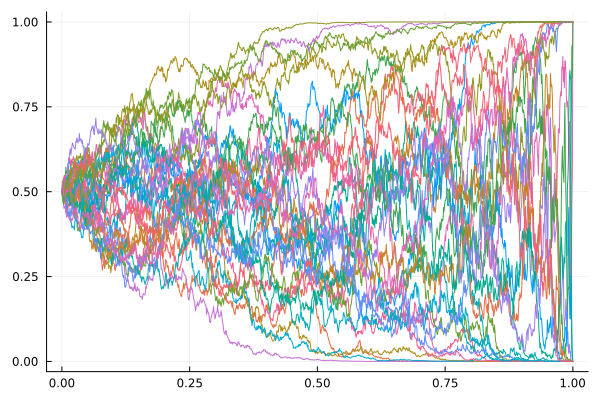}
  \captionof{figure}{Simulations of Bass martingales ($p=1$), from \cite{BBexciting}.}
  \label{fig:test1}
\end{minipage}%
\begin{minipage}{.5\textwidth}
  \centering
  \includegraphics[width=4cm,height=3cm]{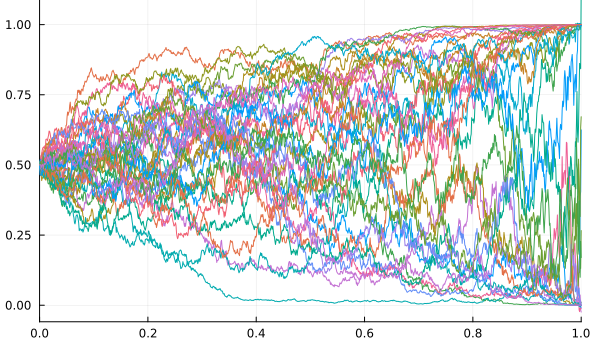}
  \captionof{figure}{Simulations of $(M_t)$, i.e.\ case  $p=1/2$.}
  \label{fig:test2}
\end{minipage}
\end{figure}

If we stretch the time-index set from $[0,1]$ to $[0,\infty)$ in a natural way, the time-changed martingale $Y_t:=M_{1-e^{-t/2}}$ admits the more amenable form
 \[dY_t=Y_t(1-Y_t)\, dW_t, \]
with $W$ a suitable Brownian motion. In this form, it can be readily interpreted in terms of filtering theory. Indeed, $Y_t$ is precisely the conditional probability of the drift being equal to $1$ for a Brownian motion with an unobservable drift which can be either 0 or 1. Moreover, the marginal distributions of $M$ and $Y$ are fully explicit and simple to describe. This is in stark contrast to the situation in \cite{BBexciting}.  If we then perform a change of space-scale $$C_t:=\log(Y_t/(1-Y_t)),$$ so that the resulting process has unit volatility coefficient, it turns out that this process satisfies the SDE
\[dC_t = \frac{1}{2}\tanh\left (\frac{C_t}{2} \right )dt + dW_t.\]
This process can also be interpreted in a number of interesting ways. For instance, the marginal laws of $C$ are a mixture between the marginal laws of a drifted Brownian motion with drift $\pm 1/2$. 
More interestingly, we have the following result, where we denote by $\mathbb W^0_{T,x}$ the law of the Brownian bridge from 0 at time 0 to $x$ at time $T$, and we define $\mathbb W_{T,\pm T/2}^0:= \frac{1}{2}\left(\mathbb W_{T,T/2}^0 + \mathbb W_{T,-T/2}^0 \right)$:\\

\emph{For every $t\in\mathbb R_+$, the law  $\mathbb W_{T,\pm T/2}^0$ restricted to $[0,t]$ converges as $T\to\infty$ to the law of $C$ restricted to $[0,t]$.}\\

This result, which we formalize in Theorem \ref{prop:bridges}, says that $C$ is precisely the law of Brownian motion $W$ conditioned on the event $W_T = \pm T/2$ as $T\to\infty$ (this is a null event of course). In other words, $C$ is a solution (more precisely, a limit of solutions) to the celebrated Schr\"odinger problem (see \cite{Le14}).
To the best of the authors' knowledge, this is the first instance that a solution to a continuous martingale transport problem has been naturally connected to the solution of a likewise continuous Schr\"odinger problem.

\subsection{Connections to the literature}

The specific relative entropy was introduced and interpreted as a rate function by Gantert \cite{gantert1991einige}. More recently \cite{BaBe24,MR4651162}  obtained an explicit formula for this quantity (between time-homogeneous Markov martingales), and F\"{o}llmer \cite{MR4433813} extended Gantert's results and established a Talagrand-type inequality between semimartingale laws using this object. In \cite{BBexciting} the specific relative entropy was used to solve an open question by Aldous (see \cite{2023arXiv230607133G} for an alternative point of view and solution). {The recent articles \cite{BeChHoLoVi24,BeChLo24} have suggested the use of the specific relative entropy as a regularization device for the problem of model calibration in finance, and studied the convergence of discrete- to continuous-time.}

 The adapted Wasserstein distance, appearing in the aforementioned Talagrand-type inequality by F\"ollmer, is a metric between stochastic processes that incorporates their temporal structure; see Definition~\ref{eq:defAW} and \cite{BBBE19,BBLZ,2021arXiv210414245B,lassalle2018causal}. Being defined as an optimization problem over a class of couplings, it is related but also different from the specific Wasserstein divergence.

Win-martingales have been proposed as models for prediction markets (cf.\ \cite{MR3096465}), and optimization problems over the set of win-martingales were proposed by Aldous \cite{aldous_winmartingale} in connection to the aforementioned open question. Such optimization problems are particular instances of martingale optimal transport, a subject that has been extensively studied in the recent years. Following \cite{BeHePe12, BoNu13,GaHeTo13,HoNe12}
martingale versions of the classical transport problem (see e.g.\ \cite{FiGl21,Sa15, Villani, Villani_Old} for recent monographs) are often considered due to applications in  mathematical finance but admit further applications, e.g.\ to the Skorokhod problem \cite{BeCoHu14, BeNuSt19}. In analogy to classical optimal transport, necessary and sufficient conditions for optimality have been established for martingale transport (MOT) problems in discrete time (\cite{BeJu21, BeNuTo16}) but not so much is known for the continuous time problem. Notable exceptions are \cite{BWZ24,BBexciting,backhoff2020martingale,2023arXiv230611019B,2023arXiv230607133G,GuLo21,huesmann2019benamou,Lo18,tan2013optimal}.

{The Schr\"{o}dinger problem has its origin in \cite{Sc31,Sc32}. In a nutshell, it asks for the most likely evolution for a large system of particles given initial and terminal configurations. By the theory of large deviation, this amounts to an entropy minimization problem. If the initial configuration is a point mass and the particles are Markovian, its solution is a Markov bridge \cite{LeRoZa14}. We refer to L\'eonard's survey \cite{Le14} for a historical account and to the more contemporary articles  \cite{Backhoff:2020aa,BeCaMaNe18,ChGePa16,Co19,PaWo18} for recent contributions. }

Finally we mention the work of Lacker \cite{lacker2020non}, extended e.g.\ in \cite{backhoff2020nonexponential,eckstein2019extended}, wherein the idea of considering a scaling limit of problems in ever higher dimension is also considered. This is more related to a large deviations principle for a particle system whose size goes to infinity, rather than to our framework of a single process that we examine at ever finer resolution.

\section{Specific Wasserstein divergence between martingales}\label{sec2}

\subsection{Specific Wasserstein divergence}

For Borel probability measures $\mu,\nu \in \pp_1(\mathcal X)$ on a metric space $(\mathcal X,d)$ we define their $1$-Wasserstein distance 
\begin{align*}
\W_1(\mu,\nu):=\inf_{\pi \in \Pi(\mu,\nu)} \int d(x,y) \, d\pi(x,y),
\end{align*}
whereby $\pp_1(\mathcal X):=\{\rho \in \pp(\mathcal X):\, \int d(x,x_0) \, d\rho(x) <\infty ,\text{ some }x_0\}$ and $\Pi(\mu,\nu)$ stands for the set of probability measures on $\mathcal X\times\mathcal X$ with first marginal $\mu$ and second marginal $\nu$. See \cite[Chapter 7]{Villani} for background. In case $\mathcal X$ is Euclidean space we will always take $d$ to be the metric associated to the Euclidean norm.

In order to define a time-consistent divergence on $\mathcal P_1(\mathbb R^N)$, with the aforementioned one-dimensional Wasserstein distance as a building block, we proceed inductively. We shall employ the following notation throughout:  If $x\in \mathbb R^N$, then $x_{i:j}:=(x_i,x_{i+1},\dots,x_j)$ for $i<j$ and $x_{i:i}=x_i$. Likewise if $\mathbb P\in \mathcal P(\mathbb R^N)$, then $\mathbb P_{i:j}$ is the law of $x_{i:j}$ under $\mathbb P$ and we denote $\mathbb P_{i}:=\mathbb P_{i:i}$. We write $\mathbb P^{x_{1:i-1}}_i$ for the conditional law of $x_i$ under $\mathbb P$ given the information of $x_{1:i-1}$, and use the convention $\mathbb P^{x_{1:0}}_1:=\mathbb P_1$.

We fix $p>0$ and define: 
{
\begin{definition}
\label{def:concatenated_D}
	Suppose $\mathbb Q,\mathbb P\in\mathcal P_1(\mathbb R)$, then $D_{\W}^{1,p}(\mathbb Q||\mathbb P):=\W_1^p(\mathbb Q,\mathbb P)$. For $N>1$, supposing that  $\mathbb Q,\mathbb P\in\mathcal P_1(\mathbb R^N)$ are such that  $P^{x_{1:i-1}}_i$ is well-defined $\mathbb{Q}_{1:i-1}$-a.s.\ for each $i \in\{1,\dots,N\}$, then we define inductively
	\[D_{\W}^{N,p}(\mathbb Q ||\mathbb P):= D_{\W}^{N-1,p}(\mathbb Q_{1:N-1} ||\mathbb P_{1:N-1})+\int \W_1^p(\mathbb Q^{x_{1:N-1}}_N , \mathbb P^{x_{1:N-1}}_N  )\, d\mathbb Q_{1:N-1}(x_{1:N-1}). \]
\end{definition}
}

Unraveling the induction, we clearly have the equivalent expression
\[D_{\W}^{N,p}(\mathbb Q ||\mathbb P) = \int_{\mathbb R^N}\left[\sum_{i=1}^{N} \W_1^p(\mathbb Q^{x_{1:i-1}}_i , \mathbb P^{x_{1:i-1}}_i  )  \right] d\mathbb Q(x).\]

\begin{remark}\label{rmk:2.2}
In the definition of $D_{\W}^{N,p}(\mathbb Q ||\mathbb P)$, the assumption that $\mathbb P^{x_{1:i-1}}_i$ is well-defined $\mathbb{Q}_{1:i-1}$-a.s.\ is necessary to integrate $\W_1^p(\mathbb Q^{x_{1:i-1}}_i , \mathbb P^{x_{1:i-1}}_i  ) $ with respect to $\mathbb{Q}_{1:i-1}$. One sufficient condition for this is that $\mathbb{Q}_{1:i-1} \ll \mathbb{P}_{1:i-1}$ for any $i=1,\dotso,N$. 

Another sufficient condition is when $\mathbb P$ is the law of a discrete-time Markov process which is uniquely defined by transition kernels $\{\mathbb P^{x_{i-1}}_i: \, i=1,\dotso, N, x_{i-1} \in \mathbb R \}$ which are defined everywhere. {Then we can set $D_{\W}^{N,p}(\mathbb Q ||\mathbb P) := \int_{\mathbb R^N}\left[\sum_{i=1}^{N} \W_1^p(\mathbb Q^{x_{1:i-1}}_i , \mathbb P^{x_{i-1}}_i  )  \right] d\mathbb Q(x)$, but one should not forget that this formula takes as an input not just $\mathbb P$ but potentially the whole of $\{\mathbb P^{x_{i-1}}_i: \, i=1,\dotso, N, x_{i-1} \in \mathbb R \}$. In particular, it could happen that $\mathbb P$ arises from two different everywhere-defined transition kernels, and this would give rise to two different expressions for $D_{\W}^{N,p}(\mathbb Q ||\mathbb P)$. Let us illustrate this in an importan particular case: Here $\mathbb P$ is the law of the process which is constant equal to zero. Then $\mathbb P^{x_{i-1}}_i = \delta_{0}$, or $\mathbb P^{x_{i-1}}_i = (1-\frac{x_{i-1}}{x_{i-1}+\text{sign}(x_{i-1})})\delta_{0} + \frac{x_{i-1}}{x_{i-1}+\text{sign}(x_{i-1})} \delta_{x_{i-1}+\text{sign}(x_{i-1})}$, or $\mathbb P^{x_{i-1}}_i = \delta_{x_{i-1}}$, are all legitimate choices and depending on the choice we would get a different value for $D_{\W}^{N,p}(\mathbb Q ||\mathbb P)$. For the purposes of this paper, we favor the latter option, and we write $\mathbb P_\delta$ whenever we mean the law of a constant martingale with the convention $\{(\mathbb P^\delta)^{x_{i-1}}_i=\delta_{x_{i-1}}: \, i=1,\dotso, N, x_{i-1} \in \mathbb R \}$. We have two main reasons for this choice. First: these are the limit of the transition kernels of the Gaussian martingale random walk $\mathbb G_\epsilon$ when we send its volatility parameter $\epsilon$ to zero. Second: as we show in Lemma~\ref{lem:valuefunctionlimit}, the correct continuous-time scaling limit of $D_{\W}^{N,p}(\mathbb Q || \mathbb G_\epsilon)$ goes  as $\epsilon\to 0$ in the sense of $\Gamma$-convergence to the correct continuous-time analogue of $D_{\W}^{N,p}(\mathbb Q ||\mathbb P_{\delta})$.  }
\end{remark}

\begin{remark}
For $p \geq 1$, thanks to the convexity of $\mathbb Q \mapsto \W_1(\mathbb Q, \mathbb P)$, it follows that $D^{N,p}_{\W}(\mathbb Q, \mathbb P)$ is also convex in $\mathbb Q$. To see this, take wlog.\ $N=2$ and two probability distributions $\mathbb{Q}, \tilde{\mathbb{Q}} \in \mathcal{P}_1(\mathbb R^2)$. Letting $t \in [0,1]$, it can be seen that $t \mathbb{Q}+(1-t) \tilde{\mathbb{Q}}$ has the disintegration 
\begin{align*}
\left( t \mathbb{Q}_1+ (1-t)\tilde{\mathbb{Q}}_1\right) \otimes \left( \frac{t \mathbb{Q}_2^{x_1}  \, d  \mathbb{Q}_1(x_1)+(1-t) \tilde{\mathbb{Q}}_2^{x_1}  \, d  \tilde{\mathbb{Q}}_1(x_1)}{t \, d  \mathbb{Q}_1(x_1)+(1-t) \, d  \tilde{\mathbb{Q}}_1(x_1)} \right),
\end{align*}
and hence  
\begin{align*}
&\int \W_1^p\left(\frac{t \mathbb{Q}_2^{x_1}  \, d  \mathbb{Q}_1(x_1)+(1-t) \tilde{\mathbb{Q}}_2^{x_1}  \, d  \tilde{\mathbb{Q}}_1(x_1)}{t \, d  \mathbb{Q}_1(x_1)+(1-t) \, d  \tilde{\mathbb{Q}}_1(x_1)}   , \mathbb P^{x_{1}}_2  \right)\, d( t \mathbb{Q}_1(x_1)+ (1-t)\tilde{\mathbb{Q}}_1(x_1)) \\
&\leq  \int \frac{t  \, d  \mathbb{Q}_1(x_1)}{t \, d  \mathbb{Q}_1(x_1)+(1-t) \, d  \tilde{\mathbb{Q}}_1(x_1)}  \W_1^p(\mathbb{Q}_2^{x_1}  || \mathbb P^{x_{1}}_2 )\, d( t \mathbb{Q}_1(x_1)+ (1-t)\tilde{\mathbb{Q}}_1(x_1)) \\
& \quad +\int \frac{(1-t)  \, d  \tilde{\mathbb{Q}}_1(x_1)}{t \, d  \mathbb{Q}_1(x_1)+(1-t) \, d  \tilde{\mathbb{Q}}_1(x_1)}  \W_1^p(\tilde{\mathbb{Q}}_2^{x_1}  || \mathbb P^{x_{1}}_2 )\, d( t \mathbb{Q}_1(x_1)+ (1-t)\tilde{\mathbb{Q}}_1(x_1)) \\
& = t \int   \W_p^p(\mathbb{Q}_2^{x_1}  || \mathbb P^{x_{1}}_2 )\, d\mathbb{Q}_1(x_1)+(1-t) \int   \W_1^p(\tilde{\mathbb{Q}}_2^{x_1}  || \mathbb P^{x_{1}}_2 )\, d\tilde{\mathbb{Q}}_1(x_1). 
\end{align*}

In the computation above, the crucial point is the convexity of $\W_1$, or more precisely, of $\W_1^p$. This is also the case for $\W_q^p$ with $1\leq q \leq p$, with $\W_q$ denoting the $q$-Wasserstein distance\footnote{Defined very much as $\W_1$, but with $d^q$ as the integrand and a power $1/q$ outside of the integral. }. In order to keep the convexity of $D^{N,p}_{\W}$ for any $p \geq 1$ we choose in this work $q=1$ throughout, but we could have easily considered $\W_q$ instead of $\W_1$.

\end{remark}

\begin{remark}\label{rmk:timeconsistency}
The construction in Definition \ref{def:concatenated_D} is the exact analogue to the Bellman principle for the conditional penalty functions of dynamic convex risk measures; see \cite[Theorem 4.5]{FoPe06}. An example of the latter is the additive decomposition of the relative entropy. More generally, if we are given a one-step risk measure $\rho_1$ with penalty function $\alpha_1$, so $ C_b(\mathbb R) \ni f\mapsto\rho_1(f):=\sup_{q\in \mathcal P(\mathbb R)}\{\mathbb \int fdq - \alpha_1(q)\}$, then by induction we can construct a risk measure $ C_b(\mathbb R^N) \ni F_N \mapsto\rho_N(F_N)$ via $\rho_N(F_N):=\rho_{N-1}((x_1,\dots,x_{N-1})\mapsto \rho_{1}(F_N(x_1,\dots,x_{N-1},\cdot)))$. Then the penalty function of $\rho_N$ is precisely\footnote{For illustration:  In the case of $N=2$ we have 
\begin{align*}\rho_2(F_2) 
=& \textstyle \sup_{q_1 \in \mathcal{P}(\mathbb R)} \left\{ \int \sup_{q^{x_1} \in \mathcal{P}(\mathbb R)}  \left\{ \int F_2(x_1,x_2) \,dq^{x_1}(x_2)-\alpha_1(q^{x_1}) \right\}\, dq_1(x_1)- \alpha_1(q_1) \right\} \\
=&  \textstyle \sup_{\mathbb Q \in \mathcal{P}(\mathbb R^2)} \left\{ \int F_2 \, d \mathbb Q - \int \alpha_1(\mathbb Q^{x_1}) \, d \mathbb Q - \alpha_1(\mathbb Q_1)\right\}.
\end{align*}
} $\mathcal P(\mathbb R^N)\ni Q\mapsto \alpha_N(Q):=\int \sum_{i=0}^{N-1} \alpha_1(Q^{x_1,\dots,x_i})\,dQ$.  A formal application of the minimax theorem indicates  
\begin{align}
\label{eq:pre_lim}\inf \rho_N(F_N(X) - (H.X)_N) = \textstyle\sup \{\int F_NdQ - \alpha_N(Q)\},  
\end{align}
where the infimum runs over predictable processes $H$ (so $(H.X)_N$ is the discrete-time stochastic integral w.r.t.\ the identity process $X$), and the supremum runs over probability measures such that $X$ is a martingale. If we would divide \eqref{eq:pre_lim} by $N$ and take $F_N(x_1,\dots,x_N):= NG(N^{-1}\sum_{i=1}^N \delta_{x_i})$ for some $G\in C_b(\mathcal P(\mathbb R))$, then the limit in $N$ of these quantities has been studied by Lacker in \cite{lacker2020non}. On the other hand, if $\alpha_1$ is the relative entropy, we divide \eqref{eq:pre_lim} by $N$ and we take $F_N= N\,F\circ p_N$, with $F\in C([0,1])$ and $p_N$ the continuous linear interpolation on $N$ points, then the limit of \eqref{eq:pre_lim} has been studied already by Cohen and Dolinsky in \cite{CoDo22} under some extra assumptions. Observe that in this case \eqref{eq:pre_lim} has the interpretation of an exponential-utility indifference price under high (if $N$ is large) risk aversion. We consider the scaling limit of \eqref{eq:pre_lim} for our choice $\alpha_1(\cdot)=\pm \mathcal W_1(\cdot,\mathbb P)^p $ as an interesting open question which we do not tackle here. Our main result (Theorem \ref{thm:convergence}) is a first step in that direction, and it pins down the right type of scaling needed.
\end{remark}

In this article, we are interested in the divergence between the distributions of two continuous martingales taking real values. Hence we consider the classical Wiener space \[\Omega:= C([0,1]; \mathbb{R})\] equipped with  its natural Borel $\sigma$-algebra. {We will denote throughout by $X$ the canonical process
\[X(\omega)=\omega,\]
and by $\langle X \rangle$ its quadratic variation process. Since we will only be dealing with martingale laws, we do not need to refer to a reference measure in order to define the quadratic variation; see Remark \ref{rem:pathwise} for more details.

}

Inspired by \cite{gantert1991einige}, we define the specific Wasserstein divergence as a scaling limit of the finite dimensional discrepancy $D^{N,p}_{\W}$. For any $\mathbb Q \in \mathcal{P}(\Omega)$ and $N \in \mathbb N$, we denote by $$(\mathbb Q)_N \in \mathcal{P}(\mathbb R^N),$$ the law $\mathbb Q$ projected on the time-grid $\{k/N : k=1,\dots,N\}$.

\begin{definition}\label{defn:SW}
For any $\mathbb Q, \mathbb P \in \mathcal{P}(\Omega)$, we define the specific $p$-Wasserstein divergence as 
\[
\SW(\mathbb Q || \mathbb P):= \liminf\limits_{\substack{N=2^n \\ n\to\infty}} N^{p/2-1} D_{\W}^{N,p}((\mathbb Q)_N || (\mathbb P)_N ),
\]
if $D_{\W}^{N,p}((\mathbb Q)_N || (\mathbb P)_N )$ is well-defined for all large enough $N=2^n$. Otherwise we set it to $+\infty$. 
\end{definition}

As the following lemma shows, if $\mathbb W^{\sigma}_{x_0}$ denotes the law of Brownian motion started at $x_0$ with instantaneous variance / volatility $\sigma^2>0$, then we have
\[D_{\W}^{N,p}((\mathbb W^{\tilde\sigma}_{x_0})_N || (\mathbb W^{\sigma}_{x_0})_N)) = N^{-p/2+1} \W^p_1\left(\mathcal{N}(0,\tilde\sigma^2), \mathcal{N}(0,\sigma^2)\right)=N^{-p/2+1}(2/\pi)^{p/2}\left||\tilde\sigma|-|\sigma|\right|^p \]
and hence 
\[\SW(\mathbb W^{\tilde\sigma}_{x_0} || \mathbb W^{\sigma}_{x_0})=(2/\pi)^{p/2}\left||\tilde\sigma|-|\sigma|\right|^p . \]
Importantly, this suggests that the scaling factor $N^{-p/2+1}$ is the right one for our purposes.

We will need the following invariance property of the divergences $D^{N,p}_{\W}$:

\begin{lemma}
Taking $T: \mathbb{R}^{N+1} \to \mathbb{R}^{N+1}, \,  (x_1,x_2,\dotso,x_{N+1}) \mapsto (x_1,x_2-x_1,\dotso,x_{N+1}-x_{N})$ 
we have the following identity for any $\mathbb{Q}, \mathbb{P} \in \mathcal{P}(\mathbb{R}^{N+1})$ 
\begin{align*}
D^{N+1,p}_{\W}(\mathbb{Q} || \mathbb{P})=D^{N+1,p}_{\W}(T(\mathbb{Q}) || T(\mathbb{P})). 
\end{align*}
\end{lemma}
\begin{proof}
Note that $T$ can also be considered as a map from $\mathbb{R}^N$ to $\mathbb{R}^N$ when restricted to the first $N$-coordinates, and hence $T(x_{1:N})$ is defined as $(x_1,x_2-x_1,\dotso, x_N-x_{N-1})$. Since $T$ is a bijection, it can be seen that 
\begin{align*}
T(\mathbb{Q})_{N+1}^{T(x_{1:N})}=(x_{N+1} \mapsto x_{N+1}+x_N)_{\#} \mathbb{Q}^{x_{1:N}}_{N+1},
\end{align*}
and therefore due to the translation invariant property of Wasserstein distance, $$\W_1^p(T(\mathbb{Q})^{T(x_{1:N})}_{N+1} ,T(\mathbb{P})^{T(x_{1:N})}_{N+1} )= \W_1^p(\mathbb{Q}^{x_{1:N}}_{N+1} ,\mathbb{P}^{x_{1:N}}_{N+1} ).$$
Now by change of measure, 
\begin{align*}
\int \W_1^p(T(\mathbb Q)^{x_{1:N}}_{N+1} , T(\mathbb P)^{x_{1:N}}_{N+1}  )\, dT(\mathbb Q)_{1:N}(x_{1:N})&= \int \W_1^p(T(\mathbb Q)^{T(x_{1:N})}_{N+1} , T(\mathbb P)^{T(x_{1:N})}_{N+1}  )\, d\mathbb Q_{1:N}(x_{1:N}) \\
&=\int \W_1^p(\mathbb Q^{x_{1:N}}_{N+1} , \mathbb P^{x_{1:N}}_{N+1}  )\, d\mathbb Q_{1:N}(x_{1:N}),
\end{align*}
and our claim follows by induction. 
\end{proof}

{

We will now fix a particular choice of martingale measure, $\mathbb P$, playing the role of a \emph{reference measure}. The reader can think of $\mathbb P$ as the law of Brownian motion, however the precise assumption that we need is as follows:

\begin{assumptions}\label{ass:P}
	Admit the existence of a jointly measurable function $\eta:[0,1]\times \mathbb R\to \mathbb R$, such that $x \mapsto \eta(t,x)$ is Lipschitz uniformly in $t$ and $\sup_{t \in [0,1 ]} |\eta(t,0)| < +\infty$.  We denote by $\mathbb P^{x_0}$ the law of the solution of the SDE $dX_t = \eta(t,X_t)dB_t$ starting from $x_0$, and if $x_0$ is fixed from the context we simply write $\mathbb P$.
\end{assumptions}

Under Assumption~\ref{ass:P}, the conditional law  $((\mathbb P)_N)_{i}^{x_{1:i-1}}$ is simply the distribution of $X_{i/N}$ where \[dX_t=\eta(t,X_t)\,dB_t, \quad X_{(i-1)/N}=x_{i-1}.\] In this case, $((\mathbb P)_N)_{i}^{x_{1:i-1}}$ is well-defined for any $x_{1:i-1} \in \mathbb{R}^{i-1}$, and for any $\mathbb Q \in \mathcal{P}_1(\Omega)$ the divergence $D_{\W}^{N,p}( (\mathbb Q)_N || (\mathbb P)_N)$ is well-defined {as discussed in Remark~\ref{rmk:2.2}. Throughout the paper this is the choice of transition kernels that we use, and we will not stress that the quantity $D_{\W}^{N,p}( (\mathbb Q)_N || (\mathbb P)_N)$ depends on this choice.}

\begin{definition}
	\label{def:mart}
	We denote by $\mathcal M^c([0,1])$ the set of continuous martingale laws with an absolutely continuous quadratic variation. The density of the quadratic variation will be denoted by $\sigma^2(t,X)$.
\end{definition}

Inspired by the developments on the particular case of the specific relative entropy \cite{gantert1991einige,MR4651162}, we have our first main result: }

\begin{theorem}\label{thm:convergence}
	Suppose $\mathbb Q \in \mathcal M^c([0,1])$ with $\sigma \in L^{p \vee 2}(\lambda \times \mathbb{Q})$, where $\lambda$ stands for the Lebesgue measure on $[0,1]$. Suppose that $\mathbb P$ satisfies Assumption~\ref{ass:P} with bounded volatility coefficient $\eta$. Then the limit inferior in the definition of $\SW(\mathbb Q || \mathbb P)$ is an actual limit, and we have 
	\[\label{eq:SW=formula}\SW(\mathbb Q || \mathbb P)=(2/\pi)^{p/2}\mathbb E_{\mathbb Q}\left [\int_0^1 \left |\, |\sigma(t,X)|-|\eta(t,X_t)|\,\right|^{p} \, dt \right ].\]
If both $\mathbb{P},\mathbb{Q}$ are time-homogeneous Markov, with Lipschitz and uniformly positive bounded volatility (i.e., $x \mapsto (\sigma(x), \eta(x))$ is Lipschitz and there exists a $\delta >0$ with $\sigma(x),\eta(x) \in (\delta, 1/\delta)$ for all $x\in \mathbb{R}$), then we have the $\mathbb Q$-a.s.\ limit
\begin{align}\label{eq:almostsure}
	\lim_{N=2^n} &N^{p/2-1} \sum_{i=1}^N \W_1^p\Big (\mathcal L_{\mathbb Q}\big(X_{i/N}|\{X_{k/N}\}_{k=1}^{i-1} \big ) , \mathcal L_{\mathbb P}\big(X_{i/N}|\{X_{k/N}\}_{k=1}^{i-1} \big) \Big ) \\ \notag
	 &=(2/\pi)^{p/2} \int_0^1 \left||\sigma(X_t)|-|\eta(X_t)|\right|^p \, dt .
	\end{align}
	
	\end{theorem}

\begin{remark}\label{rem:pathwise}
Thanks to \cite{Ka95}, there exists an adapted increasing stochastic process which is a.s.\ equal to the quadratic variation of $X$ under any martingale measure. With some abuse of notation we still denote by $\langle X\rangle  $ this process. Then we have 
\[\mathbb{Q} \mapsto \SW(\mathbb Q || \mathbb P)= (2/\pi)^{p/2}\mathbb E_{\mathbb Q}\left [\int_0^1 \left||\sqrt{d \langle X \rangle_t/ dt}|-|\eta(t,X)|\right|^{p} \, dt \right ] \]
is linear on the space of martingale measures considered in Theorem \ref{thm:convergence}. {Hence the specific Wasserstein divergence is linear even though it is the limit of non-linear objects. The same happens when one concatenates relative entropies to obtain the specific relative entropy. One should think of these linear objects as something more akin to linear transport costs than classical divergences.}
\end{remark}

\begin{remark}\label{rem:ex_reference}
	Taking $\eta$ to be a non-negative constant, say $\eta=\bar \sigma\geq 0$, Theorem \ref{thm:convergence} says
	\[\SW(\mathbb Q || \mathbb P)=(2/\pi)^{p/2}\mathbb E_{\mathbb Q}\left [\int_0^1 \left||\sigma(t,X)|-\bar\sigma\right|^{p} \, dt \right ],\]
	as promised in the introduction. Particularly natural are the choices $\bar\sigma=1$, corresponding to standard Brownian motion, and $\bar\sigma=0$, corresponding to the constant martingale.

\end{remark}

{
\begin{remark}
	If we had taken $\W_q^p$ instead of $\W_1^p$ in Definition \ref{def:concatenated_D}, then Theorem \ref{thm:convergence}  would remain true but in the r.h.s.\ of \ref{eq:SW=formula} we would get the factor $\E[|Z|^q]^{p/q}$ with $\text{Law}(Z)=\mathcal{N}(0,1)$, instead of $(2/\pi)^{p/2}$ . 
\end{remark}
}
    
\begin{proof}[Proof of Theorem \ref{thm:convergence}]
\emph{Step 1:} \quad Let us first consider the case $p \geq 2$.  With some abuse of notation, we denote by $\mathbb{Q}_k^{x_{1:k-1}}$ the conditional distribution of $k$-th marginal of $(\mathbb{Q})_N$ given the first $(k-1)$ coordinates. Then we obtain that 
\begin{align*}
D^{N,p}_{\W}((\mathbb{Q})_N || (\mathbb{P})_N)= \int_{\mathbb{R}^N} \sum_{i=1}^N\W_1^p(\mathbb{Q}_i^{x_{1:i-1}},\mathbb{P}_i^{x_{1:i-1}}) \, d (\mathbb{Q})_N(x). 
\end{align*}
Taking $\sigma_N^2=\mathbb{E}^{\lambda \times \mathbb{Q}}[\sigma^2 \, |\, \mathcal{P}_N]$, where $\mathcal{P}_N:= \sigma \left(\{(s,t] \times A : \, s<t \in \{1/N,\dotso,1\}, A \in \sigma(X_{1:s}) \}\right)$, it can be seen that 
\begin{align*}
\sigma_N^2((i-1)/N, x)&=N\mathbb{E}^{\mathbb Q} [|X_{i/N}-X_{(i-1)/N}|^2 \, | \, X_{1:i-1}=x_{1:i-1}] \\
&=N\mathbb{E}^{\mathbb{Q}}\left[\int_{(i-1)/N}^{i/N} \sigma(s,X)^2 \, ds \, | \, X_{1:i-1}=x_{1:i-1} \right],
tr\end{align*}
and $\sigma_N^2((i-1)/N,x)$ is only dependent on $x_{1:i-1}$. Also we take $\eta_N(t,x):=\eta(  (i-1)/N ,x_{i-1}  )$ for $t \in ((i-1)/N,i/N]$. 
Then by martingale convergence theorem and the continuity of $\eta$,  $\lim_{N=2^n} (\sigma_N,\eta_N)=(\sigma,\eta)$ $\lambda \times \mathbb{Q}$-a.s. 

We approximate $\mathbb{Q}_i^{x_{1:i-1}}$ and $\mathbb{P}_i^{x_{1:i-1}}$ by Gaussian distributions $\tilde{\mathbb{Q}}_{i}^{x_{1:i-1}}$ and $\tilde{\mathbb{P}}_{i}^{x_{1:i-1}}$,
\begin{align*}
\tilde{\mathbb{Q}}_{i}^{x_{1:i-1}} &= \mathcal{N}\left(x_{i-1},\frac{\sigma^2_N((i-1)/N, x)}{N}  \right),  \\
\tilde{\mathbb{P}}_i^{x_{1:i-1}} &=\mathcal{N}\left(x_{i-1}, \frac{\eta_N^2((i-1)/N,x)}{N} \right),
\end{align*}
and then $\W_1^p(\tilde{\mathbb{Q}}_{i}^{x_{1:i-1}},\tilde{\mathbb{P}}_{i}^{x_{1:i-1}})=\left( \frac{2}{\pi N}\right)^{p/2} \left| |\sigma_N((i-1)/N, x)|-|\eta_N((i-1)/N,x)| \right|^p$.
Supposing that 
\begin{align}\label{eq:normalapprox}
&\lim_{N=2^n} N^{p/2-1}\int_{\mathbb{R}^N} \sum_{i=1}^N\W_1^p(\mathbb{Q}_i^{x_{1:i-1}},\mathbb{P}_i^{x_{1:i-1}}) \, d (\mathbb{Q})_N(x) \notag \\
&=\lim_{N=2^n} N^{p/2-1}\int_{\mathbb{R}^N} \sum_{i=1}^N\W_1^p(\tilde{\mathbb{Q}}_i^{x_{1:i-1}},\tilde{\mathbb{P}}_i^{x_{1:i-1}}) \, d (\mathbb{Q})_N(x),
\end{align}
it can be seen that
\begin{align*}
\SW (\mathbb{Q} || \mathbb{P})=&\lim_{N=2^n} N^{p/2-1}\int_{\mathbb{R}^N} \sum_{i=1}^N\W_1^p(\tilde{\mathbb{Q}}_i^{x_{1:i-1}},\tilde{\mathbb{P}}_i^{x_{1:i-1}}) \, d (\mathbb{Q})_N(x) \\
&= \lim_{N=2^n} \, \left(\frac{2}{\pi}\right)^{p/2}\mathbb{E}^{\mathbb{Q}}\left[\frac{1}{N}\sum_{i=1}^N \left| |\sigma_N((i-1)/N, X)|-|\eta((i-1)/N,X)| \right|^p\right]   \\
&= (2/\pi)^{p/2}\mathbb E_{\mathbb Q}\left [\int_0^1 \left||\sigma(t,X)|-|\eta(t,X_t)|\right|^{p} \, dt \right ].
\end{align*}
The last equality follows from $L^p$ martingale convergence, the dominated convergence theorem and the fact that $(|\sigma_N|,\eta_N) \to (|\sigma|,\eta)$, $\lambda \times \mathbb{Q}$-a.s.

\vspace{4pt}

\emph{Step 2:} \quad It remains to verify \eqref{eq:normalapprox}, and we claim that
\begin{align}\label{eq:Qconverge}
\lim\limits_{N=2^n}N^{p/2-1} \int \sum_{i=1}^N \W_1^p(\mathbb{Q}_i^{x_{1:i-1}}, \tilde{\mathbb{Q}}_i^{x_{1:i-1}}) \, d(\mathbb{Q})_N(x)=0.
\end{align}
Thanks to the martingale representation theorem, on an extended filtered probability space $(\bar \Omega, (\bar{\mathcal{F}}_t)_{t \in [0,1]}, \bar{\mathbb{Q}})$, there exists a Brownian motion $(B_t)$ and an adapted process $\bar \sigma$ such that $d X_t= \bar \sigma \, d B_t$ and $\bar{\sigma}(s, \bar\omega)=|\sigma(s, X(\bar \omega))| \geq 0$, $\bar{\mathbb{Q}}$-$a.s.$  Now $\mathbb{Q}_i^{x_{1:i-1}}$ is the law $$\mathcal{L}\left(X_{(i-1)/N}+\int_{(i-1)/N}^{i/N} |\sigma(s,X)| \, dB_s \, | \, X_{1/N:(i-1)/N}=x_{1:i-1}\right)$$ while $\tilde{\mathbb{Q}}_i^{x_{1:i-1}}$ can be represented by the distribution $$\mathcal{L}\left(X_{(i-1)/N}+|\sigma_N((i-1)/N,X)|(B_{i/N}-B_{(i-1)/N}) \, | \, X_{1/N:(i-1)/N}=x_{1:i-1} \right).$$ Therefore 
\begin{align}\label{eq:coupling}
\mathcal{L} \left( \left(x_{i-1}+\int_{(i-1)/N}^{i/N} |\sigma(s,X)| \, dB_s \,,\, x_{i-1}+|\sigma_N((i-1)/N,X)|(B_{i/N}-B_{(i-1)/N}) \right) \, \big| \,  X_{1/N:(i-1)/N}=x_{1:i-1} \right) 
\end{align}
provides a natural coupling between $\mathbb{Q}_i^{x_{1:i-1}}$ and $\tilde{\mathbb{Q}}_i^{x_{1:i-1}}$, and hence using $\W_1^p \leq \W_p^p$ we get the upper bound
\begin{align*}
\W_1^p(\mathbb{Q}_i^{x_{1:i-1}}, \tilde{\mathbb{Q}}_i^{x_{1:i-1}}) &\leq \E^{\bar{\mathbb{Q}}}\left[ \left| \int_{(i-1)/N}^{i/N} |\sigma(s,X)|-|\sigma_N((i-1)/N,X)| \, dB_s \right|^p \, \big| \,X_{1/N:(i-1)/N}=x_{1:i-1}   \right]. 
\end{align*}
Integrating the above inequality over $\mathbb{Q}$,  one gets that
\begin{align*}
\int  \W_1^p(\mathbb{Q}_i^{x_{1:i-1}}, \tilde{\mathbb{Q}}_i^{x_{1:i-1}}) \, d(\mathbb{Q})_N(x) &\leq C\E^{\bar{\mathbb{Q}}} \left[ \left| \int_{(i-1)/N}^{i/N} |\sigma(s,X)|-|\sigma_N((i-1)/N,X)| \, dB_s \right|^p \right] \\
&\leq C \E^{\bar{\mathbb{Q}}}\left[\left|\int_{(i-1)/N}^{i/N} (|\sigma(s,X)|-|\sigma_N((i-1)/N,X)|)^2 \, ds\right|^{p/2}  \right] \\
& \leq \frac{C}{N^{p/2-1}} \E^{\bar{\mathbb{Q}}}\left[\int_{(i-1)/N}^{i/N} \left||\sigma(s,X)|-|\sigma_N((i-1)/N,X)|\right|^p \, ds\  \right],
\end{align*}
where we use BDG and Jensen's inequalities. Therefore, we obtain that 
\begin{align}\label{eq:upperbound1}
N^{p/2-1}\int \sum_{i=1}^N \W_1^p(\mathbb{Q}_i^{x_{1:i-1}}, \tilde{\mathbb{Q}}_i^{x_{1:i-1}}) \, d(\mathbb{Q})_N(x) \leq C\E^{\mathbb{Q}} \left[\int_0^1 \left||\sigma|-|\sigma_N|\right|^p \, ds \right].
\end{align}
Due to the $L^{p/2}$ martingale convergence theorem, $\sigma_N^2 \to \sigma^2$ with respect to $L^{p/2}$ norm, and hence $\sigma^p_N$ is uniformly integrable. Therefore, it can be easily seen that $|\sigma_N| \to |\sigma|$ in $L^p$. As a result, the right hand side converges to $0$ as $N \to \infty$, and thus we verify \eqref{eq:Qconverge}.

Similarly, according to BDG and Jensen's inequalities, it can be seen that 
\begin{align}\label{eq:momentbound}
N^{p/2-1}\int \sum_{i=1}^N \W_1^p(\mathbb{Q}_i^{x_{1:i-1}}, \mathbb{P}_i^{x_{1:i-1}}) \, d(\mathbb{Q})_N(x) \leq C\left(1+ \E^{\mathbb{Q}}\left[\int_0^1 |\sigma(s,X)|^p \, ds \right] \right), 
\end{align}
where $C$ is a constant depending on $\lVert \eta \rVert_{\infty}$. In the same way, $$N^{p/2-1}\int \sum_{i=1}^N \W_1^p(\tilde{\mathbb{Q}}_i^{x_{1:i-1}},\mathbb{P}_i^{x_{1:i-1}}) \, d(\mathbb{Q})_N(x)\leq C\left(1+ \E^{\mathbb{Q}}\left[\int_0^1 |\sigma_N(s,X)|^p \, ds \right] \right).$$ 
Therefore, applying the inequality $|a^p-b^p| \leq C|a-b|(|a|^{p-1}+|b|^{p-1})$ for $a,b \in \mathbb{R}$, we get that  
\begin{align}\label{eq:upperbound2}
&\int_{\mathbb{R}^N} \sum_{i=1}^N \left(\W_1^p(\tilde{\mathbb{Q}}_i^{x_{1:i-1}},\mathbb{P}_i^{x_{1:i-1}})-\W_1^p(\mathbb{Q}_i^{x_{1:i-1}},\mathbb{P}_i^{x_{1:i-1}}) \right)\, d (\mathbb{Q})_N(x) \\
& \leq C \int_{\mathbb{R}^N} \sum_{i=1}^N \W_1(\mathbb{Q}_i^{x_{1:i-1}}, \tilde{\mathbb{Q}}_i^{x_{1:i-1}}) \W_1^{p-1}(\mathbb{Q}_i^{x_{1:i-1}}, \mathbb{P}_i^{x_{1:i-1}})\, d (\mathbb{Q})_N(x) \notag \\
& \quad + C \int_{\mathbb{R}^N} \sum_{i=1}^N \W_1(\mathbb{Q}_i^{x_{1:i-1}}, \tilde{\mathbb{Q}}_i^{x_{1:i-1}}) \W_1^{p-1}(\tilde{\mathbb{Q}}_i^{x_{1:i-1}}, \mathbb{P}_i^{x_{1:i-1}})\, d (\mathbb{Q})_N(x).  \notag
\end{align}
Thanks to H\"{o}lder's inequality, the first term on the right is bounded by 
\begin{align*}
C \left| \int_{\mathbb{R}^N} \sum_{i=1}^N \W^p_1(\mathbb{Q}_i^{x_{1:i-1}}, \tilde{\mathbb{Q}}_i^{x_{1:i-1}})\, d (\mathbb{Q})_N(x) \right|^{1/p} \left| \int_{\mathbb{R}^N} \sum_{i=1}^N \W^p_1(\mathbb{Q}_i^{x_{1:i-1}}, \mathbb{P}_i^{x_{1:i-1}})\, d (\mathbb{Q})_N(x) \right|^{(p-1)/p}. 
\end{align*}
We have a similar estimate for the second term on the right, and hence due to \eqref{eq:Qconverge} and \eqref{eq:momentbound}, we conclude that 
\begin{align*}
\lim\limits_{N=2^n} N^{p/2-1}\int_{\mathbb{R}^N} \sum_{i=1}^N \left(\W_1^p(\tilde{\mathbb{Q}}_i^{x_{1:i-1}},\mathbb{P}_i^{x_{1:i-1}})-\W_1^p(\mathbb{Q}_i^{x_{1:i-1}},\mathbb{P}_i^{x_{1:i-1}}) \right)\, d (\mathbb{Q})_N(x)=0. 
\end{align*}
By the same reasoning, one can show that 
\begin{align*}
\lim\limits_{N=2^n} N^{p/2-1}\int_{\mathbb{R}^N} \sum_{i=1}^N \left(\W_1^p(\tilde{\mathbb{Q}}_i^{x_{1:i-1}},\mathbb{P}_i^{x_{1:i-1}})-\W_1^p(\tilde{\mathbb{Q}}_i^{x_{1:i-1}},\tilde{\mathbb{P}}_i^{x_{1:i-1}}) \right)\, d (\mathbb{Q})_N(x)=0,
\end{align*}
and thus we verify \eqref{eq:normalapprox}.

\vspace{4pt}

\emph{Step 3:} \quad Now let us prove the result for $p \in (0,2)$. Thanks to the same coupling as in \eqref{eq:coupling}, we get that 
\begin{align*}
&\int  \W_1^p(\mathbb{Q}_i^{x_{1:i-1}}, \tilde{\mathbb{Q}}_i^{x_{1:i-1}}) \, d(\mathbb{Q})_N(x)   \\
&\leq \int \E^{\mathbb{Q}}\left[\left(\int_{(i-1)/N}^{i/N} (|\sigma(s,X)|-|\sigma_N((i-1)/N,X)|)^2 \, ds\right)\, \big| \,X_{1/N:(i-1)/N}=x_{1:i-1}     \right]^{p/2}d(\mathbb{Q})_N(x)\\
& \leq \left(\int_{(i-1)/N}^{i/N} \E^{\mathbb{Q}}\left[(|\sigma(s,X)|-|\sigma_N((i-1)/N,X)|)^2 \right] ds \right)^{p/2} 
\end{align*}
where in the last inequality we use the concavity of $x \mapsto x^{p/2}$ over $\mathbb{R}_+$. 
Summing the above inequality over $i=1,\dotso, N$, and making use of $$\sum_{i=1}^N a_i^{p/2} b_i^{(2-p)/2} \leq \left(\sum_{i=1}^N a_i\right)^{p/2} \left(\sum_{i=1}^N b_i\right)^{(2-p)/2},$$
we obtain that 
\begin{align*}
N^{p/2-1}\int  \sum_{i=1}^N \W_1^p(\mathbb{Q}_i^{x_{1:i-1}}, \tilde{\mathbb{Q}}_i^{x_{1:i-1}}) \, d(\mathbb{Q})_N(x) \leq \left(\E^{\mathbb{Q}}\left[ \int_0^1 \left||\sigma|-|\sigma_N|\right|^2  ds \right]\right)^{p/2},
\end{align*}
where the right hand side converges to $0$ since $\sigma \in L^2$. Then by the same argument as in the case $p \geq 2$, we conclude the result.

\vspace{4pt}

\emph{Step 4:} \quad Finally, we prove \eqref{eq:almostsure} for the case $p>2$ and the argument for $p \in (0,2]$ is the same. It is sufficient to estimate $\E^{\mathbb{Q}} \left[\int_0^1 ||\sigma|-|\sigma_N||^p \, ds \right]$ and apply Borel-Cantelli. Without loss of generality, we assume both $\sigma$ and $\sigma_N$ are nonnegative. For each $i =1,\dotso, N$ and $s \in ((i-1)/N,i/N]$, we have 
\begin{align*}
\left|\sigma_N((i-1)/N,X)-\sigma(X_s)\right|^p \leq & 2^{p-1} \left|\sigma_N((i-1)/N,X)-\sigma(X_{(i-1)/N}) \right|^p+  2^{p-1} \left|\sigma(X_{(i-1)/N})-\sigma(X_s) \right|^p \\
\leq & C \left|\sigma^2_N((i-1)/N,X)-\sigma^2(X_{(i-1)/N}) \right|^p+  2^{p-1}\left|\sigma(X_{(i-1)/N})-\sigma(X_s) \right|^p,
\end{align*}
where in the last inequality we use the fact that $\sigma,\sigma_N >1/\delta$. For the first term on the right, it follows from the definition of $\sigma_N^2$ that 
\begin{align*}
&C \left|\sigma^2_N((i-1)/N,X)-\sigma^2(X_{(i-1)/N}) \right|^p  \leq CN^p \left|\E^{\mathbb{Q}}\left[\int_{(i-1)/N}^{i/N}\sigma^2(X_t)-\sigma^2(X_{(i-1)/N})  \,dt  \, \Big| \, X_{(i-1)/N}\right] \right|^p \\
& \leq CN^p\left| \E^{\mathbb Q}\left[\int_{(i-1)/N}^{i/N} |X_t-X_{(i-1)/N}|\,dt \, \Big| \, X_{(i-1)/N} \right]\right|^p \leq CN^p \left|\int_{(i-1)/N}^{i/N} \E^{\mathbb Q}[|X_t-X_{(i-1)/N}|  \, | \, X_{(i-1)/N}] \,dt \right|^p \\
& \leq CN^p \left|\int_{(i-1)/N}^{i/N} \sqrt{t-(i-1)/N}\,dt \right|^p\leq\frac{C}{N^{p/2}},
\end{align*}
where we use boundedness and Lipschitz property of $\sigma$. Together with the inequality 
\begin{align*}
\E^{\mathbb Q} \left[ \left|\sigma(X_{(i-1)/N})-\sigma(X_t) \right|^p\right] \leq \E^{\mathbb Q} \left[ |X_t-X_{(i-1)/N}|^p\right]\leq (t-(i-1)/N)^{p/2},
\end{align*}
we get the estimate
\begin{align*}
\E^{\mathbb{Q}} \left[\int_0^1 ||\sigma|-|\sigma_N||^p \, ds \right] \leq \frac{C}{N^{p/2}}.
\end{align*}

Thanks to \eqref{eq:upperbound1} and \eqref{eq:upperbound2}, we have  
\begin{align*}
N^{p/2-1}\int_{\mathbb{R}^N} \sum_{i=1}^N \left|\W_1^p(\tilde{\mathbb{Q}}_i^{x_{1:i-1}},\mathbb{P}_i^{x_{1:i-1}})-\W_1^p(\mathbb{Q}_i^{x_{1:i-1}},\mathbb{P}_i^{x_{1:i-1}}) \right|\, d (\mathbb{Q})_N(x)\leq \frac{C}{N^{1/2}}. 
\end{align*}
Applying Borel-Cantelli to the sequence $N=2^n$, we conclude that 
\begin{align*}
\lim_{N=2^n}N^{p/2-1}\sum_{i=1}^N \W_1^p(\tilde{\mathbb{Q}}_i^{x_{1:i-1}},\mathbb{P}_i^{x_{1:i-1}})= \lim_{N=2^n}N^{p/2-1}\sum_{i=1}^N\W_1^p(\mathbb{Q}_i^{x_{1:i-1}},\mathbb{P}_i^{x_{1:i-1}}).
\end{align*}
By the same token, 
\begin{align*}
\lim_{N=2^n}N^{p/2-1}\sum_{i=1}^N \W_1^p(\tilde{\mathbb{Q}}_i^{x_{1:i-1}},\mathbb{P}_i^{x_{1:i-1}})= \lim_{N=2^n}N^{p/2-1}\sum_{i=1}^N\W_1^p(\tilde{\mathbb{Q}}_i^{x_{1:i-1}},\tilde{\mathbb{P}}_i^{x_{1:i-1}}).
\end{align*}
Therefore we get that 
\begin{align*}
	\lim_{N=2^n} N^{p/2-1} \left[\sum_{i=1}^N \W_1^p(\mathbb Q^{x_{1:i-1}}_i , \mathbb P^{x_{1:i-1}}_i  ) \right]= \int_0^1 ||\sigma(x)|-|\eta(x)||^p \, dt \quad \text{for $\mathbb{Q}$-a.e. $x$}
\end{align*}
\end{proof}

Let us discuss in detail how Theorem \ref{thm:convergence} relates to the literature. The only precursor that we are aware of is the case of the specific relative entropy. In that case, that a scaling limit of relative entropies is greater than or equal to an explicit function of the quadratic variation, was already obtained by Gantert in \cite[Satz 1.3]{gantert1991einige} and subsequently refined in recent times by F\"ollmer in \cite{MR4433813}. That equality can occur in that case, was obtained under strong assumptions in \cite{MR4651162}. Compared to these results, we obtained the equality in Theorem \ref{thm:convergence} under assumptions that are vastly weaker. This is possible because controlling the error caused by approximating conditional distributions of $\Q,\mathbb{P}$ over short time-intervals, by Gaussians measures with the same mean and variance, is significantly more demanding in the case of the relative entropy. 

\subsection{Relation between specific relative entropy and adapted Wasserstein distance}\label{subsec:inequalities}

Let us provide the definition of specific relative entropy and adapted Wasserstein distance. 

\begin{definition}
Let $H$ be the relative entropy defined via $H(\mu || \nu):= \int \log(d\mu/d\nu) \, d \mu$ with $\mu,\nu \in \mathcal{P}(\mathcal{X})$. Then the specific relative entropy is defined as the limit
\begin{align*}
h(\Q || \PP):= \liminf\limits_{\substack{N=2^n \\ n \to \infty}} \frac{1}{N}H((\Q)_N || (\PP)_N).
\end{align*}
\end{definition}

With our methodology of defining divergences between processes in the introduction, $h$ is exactly equal to the limit of $c^H_N D^N_H$ with $c_N^H:=\frac{1}{N}$. Suppose $\Q,\PP$ are martingale measures with volatility $\sigma,\eta$ respectively. Under some strong conditions \cite{MR4651162} obtains explicit the formula 
\begin{align*}
h(\Q || \PP)=\frac{1}{2} \E_{\Q} \left[\int_0^1 \left\{ \frac{\sigma(M_t)^2}{\eta(M_t)^2}-1-\log \frac{\sigma(M_t)^2}{\eta(M_t)^2} \right\} \, dt \right],
\end{align*}
while in general the l.h.s.\ is the greater one (\cite[Satz 1.3]{gantert1991einige}).


\begin{definition}[Bicausal coupling]
Let $\mathbb{Q}$, $\mathbb{P}$ be two probability distributions over $\Omega=C([0,1];\mathbb{R})$. We denote by $(X,Y)$ the canonical process on $\Omega \times \Omega$, and by $(\mathcal{F}^X_t)_{t \in [0,1]}$, $(\mathcal{F}^Y_t)_{t \in [0,1]}$ canonical filtrations generated by $X$, $Y$ respectively. Then a probability measure $\pi \in \mathcal{P}_2(\Omega \times \Omega)$ is said to be a bicausal coupling between $\mathbb{Q}$ and $\mathbb{P}$ if  
\begin{enumerate}
\item $\pi(A \times  \Omega)=\mathbb{Q}(A)$, $\pi(\Omega \times B)=\mathbb{P}(B)$ for all $A,B \in \mathcal{B}(\Omega)$. 
\item Causal from $\mathbb{Q}$ to $\mathbb{P}$: under $\pi$,  ${\mathcal{F}^X_1} \indep_{\mathcal{F}^X_t} {\mathcal{F}^Y_t}$ for all $t \in [0,1]$. 
\item Causal from $\mathbb{P}$ to $\mathbb{Q}$: under $\pi$,  ${\mathcal{F}^Y_1} \indep_{\mathcal{F}^Y_t} {\mathcal{F}^X_t}$ for all $t \in [0,1]$.
\end{enumerate}
We denote the set of all bicausal couplings between $\mathbb{Q}$ and $\mathbb{P}$ by $\Pi_{bc}(\mathbb{Q},\mathbb{P})$. 
\end{definition}

{Observe that $\Pi_{bc}(\mathbb{Q},\mathbb{P})$ is always non-empty, since the independent product of $\mathbb{Q}$ and $\mathbb{P}$ always lies in this set. The reader might find useful to think of property (2) above in words: in order to predict $\mathcal{F}^Y_t$ given the information in $\mathcal{F}^X_1$, it is enough to look at $\mathcal{F}^X_t$. This is in particular the case if the process $Y$ is adapted to the filtration generated by the process $X$.}

The concept of bicausal coupling is a natural extension of coupling between probability distribution to the framework of stochastic processes, in which the filtration is a crucial component. See \cite{lassalle2018causal,ABZ} and the references therein for more on this concept. With this notion, one can define the so-called adapted Wasserstein distance between stochastic processes, which has been used in stability analysis for various stochastic optimization problems \cite{PflugPichler,Pflug,BBLZ,ABZ,BBBE19,bartl2023sensitivity}. 

\begin{definition}[$\mathcal{AW}_2$]\label{eq:defAW}
Letting $\mathbb{Q}, \mathbb{P} \in \mathcal{P}_2(\Omega)$ be two distributions of martingales, define 
\begin{align*}
\mathcal{AW}_2(\mathbb{Q},\mathbb{P})^2:= \inf_{\pi \in \Pi_{bc}(\mathbb{Q},\mathbb{P})}\E^{\pi}\left[ |X_1-Y_1|^2\right] 
\end{align*}
\end{definition}

With these ingredients we can state a chain of (in)equalities recently derived by F\"ollmer in \cite{MR4433813}. Our proof method, based on time discretization, differs from that author's approach.

\begin{proposition}\label{prop:inequalities}
Suppose $p=2$ and $\mathbb Q \in \mathcal M^c([0,1])$ with $\sigma \in L^{2}(\lambda \times \mathbb{Q})$. Denote by $\mathbb W$ the Wiener measure on $C([0,1];\mathbb R)$. Then we have the (in)equalities 
\begin{align*}
\frac{1}{2} \mathcal{AW}_2(\mathbb{Q},\mathbb{W})^2= \frac{1}{2} \mathcal{SW}_2(\mathbb{Q} || \mathbb{W})  \leq h( \mathbb{Q} || \mathbb{W}).
\end{align*}
\end{proposition}

\begin{proof}
According to Theorem~\ref{thm:convergence}, to prove the first equality, it suffices to show that 
\begin{align*}
\mathcal{AW}_2(\mathbb{Q},\mathbb{W})^2= \mathbb{E}_{\mathbb{Q}}\left[\int_0^1 (|\sigma(t,X)|-1)^2 \,dt \right].
\end{align*}
Suppose $\pi \in \Pi_{bc}(\mathbb{Q}, \mathbb{W})$. Thanks to the bicausal condition, $(X_t,Y_t)_{0\leq t \leq 1}$ is a martingale with respect to the filtration $\left(\mathcal{F}_{t,t}^{X,Y}\right)_{0 \leq t \leq 1}$ under $\pi$. Then due to the martingale representation theorem, there exist two independent Brownian motions $W^1, W^2$ (perhaps in an enlarged probability space) such that 
\begin{align*}
dX_t&= \tilde{\sigma}^1_t \, dW^1_t + \tilde{\sigma}^2_t \, dW^2_t, \\
dY_t&= \tilde{\eta}^1_t \, dW^1_t+\tilde{\eta}^2_t \, dW^2_t,
\end{align*}
with constraints 
$|\tilde{\sigma}^1_t|^2+|\tilde{\sigma}^2_t|^2=|\sigma(t,\cdot)|^2$, $|\tilde{\eta}^1_t|^2+|\tilde{\eta}^2_t|^2=1$ $\pi$-a.e. Then by It\^{o}'s isometry, 
\begin{align*}
\E_{\pi}\left[ |X_1-Y_1|^2\right]=&\E_{\pi} \left[\int_0^1 |\tilde{\sigma}^1_t-\tilde{\eta}^1_t|^2+|\tilde{\sigma}^2_t-\tilde{\eta}^2_t|^2 \,dt \right] \\
=& \E_{\pi} \left[ \int_0^1 |\sigma(t,X)|^2+1- 2\tilde{\sigma}_t^1\tilde{\eta}_t^1-2\tilde{\sigma}_t^2\tilde{\eta}_t^2 \, dt \right] \\
\geq & \E_{\pi} \left[ \int_0^1 |\sigma(t,X)|^2+1- 2|\sigma(t,X)| \, dt \right]=\mathbb{E}_{\mathbb{Q}}\left[\int_0^1 (|\sigma(t,X)|-1)^2 \,dt \right],
\end{align*}
where we use Cauchy-Schwarz in the last inequality. Moreover, it is clear that the equality is obtained when $\tilde{\sigma}_t^1=\sigma(t,\cdot)$, $\tilde{\eta}^1_t={\text{sign}(\sigma(t,\cdot))}$, and $\tilde{\sigma}_t^2=\tilde{\eta}_t^2=0$. 

Let us now prove the inequality involving specific relative entropy invoking the well-known Talagrand inequality for standard Gaussian distribution \cite[Theorem 1.5]{MR2895086}
\begin{align*}
\mathcal{W}_1(\mu,\mathcal{N}(x,\sigma^2)) \leq  \sigma \sqrt{2 H(\mu || \mathcal{N}(x,\sigma^2))} \quad \quad \text{for all $x \in \mathbb{R}, \, \mu \in \mathcal{P}(\mathbb{R})$}. 
\end{align*}
Recall that $(\mathbb{Q})_N$ is the projection of $\mathbb{Q}$ on the time-grid $\{k/N: k=1,\dotso,N\}$, and $\mathbb{Q}_k^{x_{1:k-1}}$ the conditional distribution of the $k$-th marginal of $(\mathbb{Q})_N$ given the first $(k-1)$ coordinates. Then it is straightforward that
\begin{align*}
D^{N,2}_{\W}((\mathbb{Q})_N || (\mathbb{W})_N)=&\int_{\mathbb{R}^N} \sum_{i=1}^N\W_1^2(\mathbb{Q}_i^{x_{1:i-1}},\mathbb{W}_i^{x_{1:i-1}}) \, d (\mathbb{Q})_N(x) \\
=&\int_{\mathbb{R}^N} \sum_{i=1}^N\W_1^2(\mathbb{Q}_i^{x_{1:i-1}},\mathcal{N}(x_{i-1}, 1/N)) \, d (\mathbb{Q})_N(x) \\
\leq &\frac{2}{N}\int_{\mathbb{R}^N} \sum_{i=1}^N H (\mathbb{Q}_i^{x_{1:i-1}} || \mathcal{N}(x_{i-1}, 1/N)) \, d (\mathbb{Q})_N(x)=\frac{2}{N} H((\mathbb{Q})_N || (\mathbb{W})_N).
\end{align*}
Letting $N \to \infty$ and using Definition~\ref{defn:SW}, we get that $ \mathcal{SW}_2(\mathbb{Q} || \mathbb{W}) \leq 2 h(\mathbb{Q}|| \mathbb{W})$. 

\end{proof}

{
\begin{remark}
   We do not explore a possible extension of Proposition \ref{prop:inequalities} to the case when the Wiener measure $\mathbb W$ is replaced by another reference measure $\mathbb P$ (satisfying our Assumption \ref{ass:P}). If the transition kernels of $\mathbb P$ would satisfy a suitable Talagrand inequality, the computations at the end of the proof would remain much the same. The actual difficulty lies at the beginning of the proof, where the idea is to compare $\mathcal{AW}_2(\mathbb{Q},\mathbb{P})^2$ and $ \mathbb{E}_{\mathbb{Q}}\left[\int_0^1 (|\sigma(t,X)|-|\eta(t,X_t)|)^2 \,dt \right]$. Our Cauchy-Schwarz argument seems to only be helpful when $\eta $ is a constant.
\end{remark}
}

\begin{remark}\label{rmk:toposame}
    For a certain class of SDEs,  $\mathcal{AW}_2$ and $\mathcal{SW}_2$ induce the same topology. More precisely, denote by $\mathcal{P}^L$ the set of laws of the solutions to $dX_t=\sigma(X_t) \, dB_t$ with $X_0=0$ and $\sigma$ being $L$-Lipschitz and such that $|\sigma(0)|\leq L$, by $\tau_{A}$ the topology of $\mathcal{AW}_2$-convergence, and by  $\tau_{S}$ the topology generated by the open balls  $U(\mathbb P,r):=\{\mathbb Q \in \mathcal{P}^L: \, \mathcal{SW}_2( \mathbb Q \| \mathbb P ) <r\}$ for any $r >0$ and $\mathbb P \in \mathcal{P}^L$. Then, we have that $\tau_A=\tau_S$ on $\mathcal{P}^L$.

Denote by $\tau$ the topology on $\mathcal{P}^L$ induced by the distance
\begin{align*}
    (\mathbb Q,\mathbb P)\mapsto \mathbb E\left[\sup_{0 \leq t\leq 1} |X_t-Y_t|^2\right]^{1/2},
\end{align*}
where $(X,Y)$ is the synchronous coupling, i.e. $X$ and $Y$ are driven by a same Brownian motion, with the given marginals.
It has been proved in \cite[Theorem 1.5]{BKR22} that $\tau$ is compact. Therefore any Hausdorff topology which is coarser than $\tau$ must coincide with $\tau$. It is clear that $\tau_{A}  \subset \tau$, $\tau_A$ is Hausdorff, and hence $\tau=\tau_A$. Let us show that $\tau_S$ Hausdorff. Indeed taking any two $\mathbb P, \mathbb P' \in \mathcal{P}^L$ with corresponding nonnegative volatilities $\sigma,\sigma'$ respectively, there exists a small enough radius $r$ such that $U(\mathbb P,r) \cap U(\mathbb P',r) = \emptyset$. Otherwise, for any $r>0$, there exists $\mathbb Q^r$ such that $\mathcal{SW}_2(\mathbb Q^r \| \mathbb P ) <r $ and $\mathcal{SW}_2(\mathbb Q^r \| \mathbb P') <r$. By the triangle inequality, we have that for any $r>0$
\begin{align*}
    \mathbb E^{\mathbb Q^r} \left[\int_0^1 (\sigma(X_t)-\sigma'(X_t))^2 \, dt   \right]^{1/2}< 2r.
\end{align*}
Due to the compactness of $\tau$, there exists a limit $\mathbb Q^*$ such that $ \mathbb E^{\mathbb Q^*} \left[\int_0^1 (\sigma(X_t)-\sigma'(X_t))^2 \, dt   \right]=0$, which implies $\sigma=\sigma'$ and $\mathbb P =\mathbb P'$. To conclude, it is sufficient to show that $\tau_S \subset \tau_A$, since we already know that $\tau_A \subset \tau$.

Indeed, suppose that $\mathbb Q \sim dX_t=\sigma(X_t) \, dB_t$ and $\mathbb P \sim dY_t=\eta(Y_t) \, dW_t$, and without loss of generality, both $\sigma$ and $\eta$ are nonnegative. We conclude by the inequalities below 
\begin{align*}
    & \mathcal{SW}_2(\mathbb Q \| \mathbb P) = \sqrt{\mathbb E^{\mathbb Q} \left[\int_0^1 (\sigma(X_t)-\eta(X_t))^2 \,dt  \right]} \\
    & \leq  \inf_{\pi \in \Pi_{bs}(\mathbb Q,\mathbb P)} \left( \sqrt{\mathbb E^{\pi} \left[\int_0^1 (\sigma(X_t)-\eta(Y_t))^2 \,dt  \right]}+\sqrt{\mathbb E^{\pi} \left[\int_0^1 (\eta(X_t)-\eta(Y_t))^2 \,dt  \right]} \right) \\
    & \leq   \inf_{\pi \in \Pi_{bs}(\mathbb Q,\mathbb P)} \left( \sqrt{\mathbb E^{\pi} \left[|X_1-Y_1|^2 \right]}+L \sqrt{\mathbb E^{\pi} \left[\int_0^1 |X_t-Y_t|^2 \,dt  \right]} \right) \leq (1+L) \mathcal{AW}_2(\mathbb Q, \mathbb P).
\end{align*}
\end{remark}

\section{Optimal win-martingales}

Win-martingales appear naturally as (idealized) models for prediction markets (cf.\ \cite{MR3096465}). A win-martingale is supposed to track the probability of an event happening at time $1$. Hence they are supposed to start with a known value in $(0,1)$ and terminate distributed as a Bernoulli random variable.

Optimization problems over the set of win-martingales were proposed by Aldous \cite{aldous_winmartingale}, and two such problems were solved in \cite{BBexciting,2023arXiv230607133G}.

\subsection{Specific Wasserstein divergence optimization over win martingales}

Given $\mu,\nu\in\mathcal P_1(\mathbb R)$ in convex order, martingale optimal transport problems in continuous-time often take the form:
\begin{align}\label{def:MOT}
\inf/\sup\left\{\E_{\Q}\left[ \int_0^1 c\left(t,X_t, \Sigma_t \right) dt \right ]: \Q\in \mathcal M^c([0,1]),\,X_0(\Q)=\mu, X_1(\Q)=\nu \right\},
\end{align}
where $\mathcal M^c([0,1])$ denotes the set of continuous martingale laws with an absolutely continuous quadratic variation, $X$ stands for the canonical process, and $\Sigma_t=d\langle X\rangle_t/dt$ for the density of its quadratic variation { that is sometimes denoted by $\sigma_t^2=\sigma^2(t,\cdot)$ as in Definition~\ref{def:mart}}. Martingale optimal transport problem is a variant of optimal transport in mathematical finance and is an essential tool for robust pricing and hedging; see e.g.\ \cite{MR3456332,MR3256817,Lo18,GuLo21}.

In this paper, we consider optimization problems among an important subclass of martingales, the so-called \emph{win-martingales}. We  write  $\mathcal M^c_{x_0}$ for the set of laws of continuous martingales  with time-index set $[0,1]$ which  have  absolutely continuous quadratic variation and start in $x_0$. The subset $\mathcal M^c_{x_0, \text{win}}$ of  \emph{win-martingales}   consist of those martingales in  $\mathcal M^c_{x_0}$ which  terminate in either $0$ or $1$. It is clear that the terminal distribution of such win-martingales is Bernoulli($x_0)$. 

Let $x_0\in (0,1)$. In \eqref{def:MOT}, taking $\mu=\delta_{x_0}$, $\nu=\text{Bernoulli}(x_0)$, and $c(t,x,\Sigma):=\Sigma^{p/2}$, the martingale optimal transport problem can be interpreted as a specific Wasserstein divergence optimization problem. We are interested in solving  for all\footnote{Except for the case $p=2$, which is trivial in that every feasible martingale is optimal.} $p>0$:
\begin{align}\label{eq:inf_divergence_prob}
\text{OPT}(p,x_0)&=\inf/\sup\left\{\SW(\mathbb{Q}|| \mathbb P_{\delta}): \Q\in \mathcal M^c_{x_0, \text{win}} \right\}  \notag \\
&\propto\inf/\sup\left\{\E^{\Q}\left[\int_0^1 \Sigma_t^{p/2}\,dt \right]: \Q\in \mathcal M^c_{x_0, \text{win}} \right\},
\end{align}
whereby we recall that $\mathbb P_{\delta}$ stands for the constant martingale (see Remark \ref{rem:ex_reference}).

First we observe that the maximization problem is trivial in the case of $p>2$ (and the same for  the minimization problem when $p \in (0,2)$) as the following example reveals. Therefore when referring to $\text{OPT}(p,x_0)$, we  solve the minimization problem if $p >2$ and the maximization problem if $p \in (0,2)$. 

\begin{example}\label{ex:infty}
Fix $x_0 \in (0,1)$, $p>2$, and take an arbitrary $\mathbb{P} \in \mathcal{M} ^c_{x_0,\text{win}}$. We construct a sequence of $\mathbb{P}^n \in \mathcal{M} ^c_{x_0,\text{win}}$ which is the distribution of 
\begin{align*}
X^n_t=
\begin{cases}
x_0, & t \in [0,1-1/n] \\
X_{n(t-1+1/n)}, & t \in [1-1/n,1],
\end{cases}
\end{align*}
where $(X_t)_{t \in [0,1]}$ is a continuous-time process with distribution $\mathbb{P}$. Then it can be easily seen by Jensen's inequality
\begin{align*}
\mathbb{E}^{\mathbb{P}^n}\left[\int_0^1 \Sigma^{p/2}_t \,dt \right]=\mathbb{E}^{\mathbb{P}^n}\left[\int_{1-1/n}^1 \Sigma^{p/2}_t \,dt \right] \geq n^{p/2-1}\mathbb{E}^{\mathbb{P}^n}\left[\int_{1-1/n}^1 \Sigma_t \,dt \right]^{p/2}= n^{p/2-1}(1-x_0)^{p/2}x_0^{p/2},
\end{align*}
and hence $$\sup\left\{\E^{\Q}\left[\int_0^1 \Sigma_t^{p/2}\,dt \right]: \Q\in \mathcal M^c_{x_0, \text{win}} \right\}=+\infty.$$ 
\end{example}

{

Second, we show that 
\begin{align*}
\text{OPT}^{\epsilon}(p,x_0)&=\inf/\sup\left\{\SW(\mathbb{Q}|| \mathbb P_{\epsilon}): \Q\in \mathcal M^c_{x_0, \text{win}} \right\}  \notag \\
&\propto\inf/\sup\left\{\E^{\Q}\left[\int_0^1 ||\sigma_t|-\epsilon|^{p/2}\,dt \right]: \Q\in \mathcal M^c_{x_0, \text{win}} \right\},
\end{align*}
converges to $\text{OPT}(p,x_0)$ as $\epsilon \to 0$, where $\mathbb P_{\epsilon}$ denotes the law of the scaled Wiener process $(\epsilon W_t)_t$, as mentioned in the introduction. This provides a reason for analyzing Problem  $\text{OPT}(p,x_0)$. Indeed, Problem $\text{OPT}^{\epsilon}(p,x_0)$ is arguably the more natural one (divergence optimization given that the reference model is a scaled Brownian motion), but we do not know how to solve it. Instead, we can approximate this problem (as $\epsilon\to 0$) with $\text{OPT}(p,x_0)$, which is more tractable.

\begin{lemma}\label{lem:valuefunctionlimit}
\begin{itemize}
    \item[1)] Suppose $p \in (0,2)$. For any $\epsilon,r>0$ and $\mathbb Q \in \mathcal{M}^c([0,1])$, it holds 
    \begin{align}\label{eq:limsup}
        \SW(\mathbb Q || \mathbb P_{\epsilon})-(2/\pi)^{p/2} \epsilon^p &\leq \SW(\mathbb Q || \mathbb P_{\delta}) \notag \\
        &\leq (1+r^2)^{p/2} \SW(\mathbb Q || \mathbb P_{\epsilon})+(2/\pi)^{p/2}|1-1/r^2|^{p/2}\epsilon^p. 
    \end{align}
    \item[2)] Suppose $p \in (2,\infty)$. Then $\SW(\cdot \, || \mathbb P_{\epsilon})$ $\Gamma$-converges to $\SW(\cdot \, || \mathbb P_{\delta})$ as $\epsilon \to 0$ in the weak topology of probability measures on the continuous path space.


    \item[3)] $\lim\limits_{\epsilon \to 0}\text{OPT}^{\epsilon}(p,x_0)=\text{OPT}(p,x_0)$ for any $p \in (0,2) \cup (2,+\infty)$ and $x_0 \in (0,1)$.
\end{itemize}   
\end{lemma}

We recall that a sequence of functionals $F_n:E\to \mathbb R$ is said to $\Gamma$-converge to $F:E\to\mathbb R$, where $E$ is a metrizable topological space, if     \begin{enumerate}
    \item[a)] 
    $\liminf\limits_{n \to \infty}F_n(e_n)\geq F(e)$ whenever $e_n\to e$, and
    \item[b)] For any $e\in E$ there is some sequence $e_n\to e$ such that $\limsup\limits_{n \to \infty} F(e_n)\leq F(e)$.
    \end{enumerate}

    \begin{proof}
It can be verified that $||\sigma_s|-\epsilon|^p \leq |\sigma_s|^p+\epsilon^p$ for any $p >0$. Indeed, $||\sigma_s|-\epsilon|^p \leq \epsilon^p$ if $|\sigma_s| \in [0,2\epsilon]$, and $||\sigma_s|-\epsilon|^p \leq |\sigma_s|^p$ if $|\sigma_s| \geq 2 \epsilon$.  It follows that $$ \SW(\mathbb Q || \mathbb P_{\epsilon})-(2/\pi)^{p/2} \epsilon^p \leq \SW(\mathbb Q || \mathbb P_{\delta}).$$ 
For $p \in (0,2)$, $\mathbb R_+ \ni x \mapsto x^{p/2}$ is sublinear, and therefore for any $r>0$ we have
  \begin{align*}
      |\sigma_s|^p= & |  (|\sigma_s|-\epsilon)^2-\epsilon^2+2\epsilon (|\sigma_s|-\epsilon)  |^{p/2} \leq \left| (1+r^2)(|\sigma_s|-\epsilon)^2+ (1/r^2-1)\epsilon^2      \right|^{p/2} \\
      \leq &(1+r^2)^{p/2}||\sigma_s|-\epsilon|^p+|1/r^2-1|^{p/2}\epsilon^p,
  \end{align*}
from which it follows that
\begin{align*}
   (\pi/2)^{p/2} \SW(\mathbb Q || \mathbb P_{\delta})=&\mathbb E^{\mathbb Q}\left[\int_0^1 |\sigma_s|^p \,ds \right] \leq (1+r^2)^{p/2} \mathbb E^{\mathbb Q}\left[\int_0^1 ||\sigma_s|-\epsilon|^p \,ds \right]+|1/r^2-1|^{p/2}\epsilon^p \\
   =&(\pi/2)^{p/2}(1+r^2)^{p/2} \SW(\mathbb Q || \mathbb P_{\epsilon})+|1/r^2-1|^{p/2}\epsilon^p.
\end{align*}
This proves Point (1). 

For Point (2) we first observe that condition (b) for $\Gamma$-convergence is trivially satisfied: given $\mathbb Q$ we take $ \mathbb Q_\epsilon:=\mathbb Q$ and check that $\SW(\mathbb Q || \mathbb P_{\delta})\geq \limsup_{\eps\to 0}\SW(\mathbb Q || \mathbb P_{\epsilon})$. To wit, if $\infty >\SW(\mathbb Q || \mathbb P_{\delta})$, we may use $||\sigma_s|-\epsilon|^p \leq |\sigma_s|^p+\epsilon^p$ together with dominated convergence to conclude. We now establish condition (a) for $\Gamma$-convergence.  By \cite[Lemma 8.4]{BaPa22}, if $\mathbb Q_{\epsilon} \to \mathbb Q$ in weak topology, it holds that
\begin{align*}
 \mathbb E^{\mathbb Q}\left[\int_t^1 |\sigma_s|^p \, ds \right] \leq    \liminf\limits_{\epsilon \to 0} \mathbb E^{\mathbb Q^{\epsilon}}\left[\int_t^1 |\sigma_s|^p \, ds \right],
\end{align*}
where $\sigma^2$ denotes the density of the quadratic variation of the canonical process.
For any integer $K \geq 2$, $||\sigma_s|-\epsilon|^p \geq ((K-1)/K)^p|\sigma_s|^p$ if $|\sigma_s| \geq K\epsilon$. Therefore we have that 
\begin{align*}
    ||\sigma_s|-\epsilon|^p \geq \left(\frac{K-1}{K}\right)^p |\sigma_s|^p- K^p\epsilon^p,
\end{align*}
and thus 
\begin{align*}
   \left(\frac{K-1}{K}\right)^p  \liminf\limits_{\epsilon \to 0} \mathbb E^{\mathbb Q^{\epsilon}}\left[\int_t^1 |\sigma_s|^p \, ds \right] \leq \liminf\limits_{\epsilon \to 0} \mathbb E^{\mathbb Q^{\epsilon}}\left[\int_t^1 ||\sigma_s|-\epsilon|^p \, ds \right].
\end{align*}
Letting $K \to +\infty$, we get that $ \mathbb E^{\mathbb Q}\left[\int_t^1 |\sigma_s|^p \, ds \right]\leq \liminf\limits_{\epsilon \to 0} \mathbb E^{\mathbb Q^{\epsilon}}\left[\int_t^1 ||\sigma_s|-\epsilon|^p \, ds \right]$. This finishes the proof of Point (2).

In order to establish Point (3), we first examine the set $\mathcal{M}^c_{x_0,\text{win}}$. 
Observe that
\begin{align*}
    \mathbb E_{\mathbb Q}[|X_t-X_s|^{2p}]\leq c_p \mathbb E_{\mathbb Q}\left[\left(\int_s^t|\sigma_u|^2 du\right)^{p}\right] &\leq c_p \mathbb E_{\mathbb Q}\left[\int_s^t\sigma_u^{2p} du\right]|t-s|^{p-1} ,
\end{align*}
where the first inequality is by BDG ($c_p$ is a universal constant) and the second one by Jensen.  As $p-1>1$ if $p>2$, we conclude by the Kolmogorov-Chentsov tightness criterion (see \cite[Corollary 16.9]{Kallbook}) that the set $\{\mathbb Q\in\mathcal{M}^c_{x_0,\text{win}}: \mathcal{SW}_p(\mathbb Q|| \mathbb P_\delta)\leq R\}$ is relatively compact, for any $R>0$, for the weak topology of probability measures on continuous path-space. By triangle inequality, 
$$ \SW(\mathbb Q || \mathbb P_{\epsilon})^{1/p }\geq \SW(\mathbb Q || \mathbb P_{\delta})^{1/p}-(2/\pi)^{1/2} \epsilon, $$
so $\{\mathbb Q\in\mathcal{M}^c_{x_0,\text{win}}: \mathcal{SW}_p(\mathbb Q|| \mathbb P_\eps)\leq R, \textnormal{ for some }\epsilon\in(0,1)\}$ is relatively compact in the same sense. We conclude that $\{\mathbb Q\in\mathcal{M}^c_{x_0,\text{win}}:\liminf_{\epsilon\to 0  }\mathcal{SW}_p(\mathbb Q|| \mathbb P_\eps)\leq R\}$ is also relatively compact in the same sense. This fact, together with Point (2), the lower-semicontinuity part of the result \cite[Theorem 8.3]{BaPa22}, and the continuity of the marginal constraints w.r.t.\ weak convergence, readily shows Point (3) for the case of $p>2$. In the case of $p\in(0,2)$, we only need to take supremum over $\mathbb Q\in\mathcal{M}^c_{x_0,\text{win}}$ in \eqref{eq:limsup}, then liminf / limsup as $\epsilon\to 0$, and then take $r\to 0$ therein.
      
    \end{proof}

}

\subsection{Ansatz for the optimizer}\label{subsec:Ansatz}

In this subsection, we propose a candidate optimizer, and verify that it is indeed the unique optimizer in the next section. The key ingredient is a first order condition for MOT obtained in \cite{BBexciting} but that we recall here for the convenience of the reader:

\begin{lemma}\label{eq:FirstOrderMOT}[First order condition for MOT on the line] 
Consider the MOT problem \eqref{def:MOT}, and suppose that $c$ is differentiable in its last variable, that $\Q$ is an optimizer, and that 
$$t\mapsto L_t:= \Sigma_t\partial_\Sigma c(t,X_t,\Sigma_t)-c(t,X_t,\Sigma_t),$$ is a continuous $\Q$-semimartingale. Then $(L_t)_{t\in[0,1)}$ is a  martingale under $\Q$.
\end{lemma}

Suppose that the optimizers of $\text{OPT}(p,x_0)$ are Markov diffusions with volatility function $\sigma: [0,1] \times [0,1] \to \mathbb{R}$. Applying Lemma~\ref{eq:FirstOrderMOT} to our case $c(t,X_t,\Sigma_t)=\Sigma_t^{p/2}$, being an optimizer implies that $\Sigma_t^{p/2}=\sigma^p(t,X_t)$ is a martingale, and hence due to It\^{o}'s formula we get an equality
\begin{align}\label{eq:pmd}
0=\partial_t \sigma^p+ \frac{1}{2} \sigma^2 \Delta \sigma^{p},
\end{align}
which is then equivalent to $$0=\partial_{t} \tilde\sigma+\frac{p-2}{2p} \Delta \tilde \sigma^{\frac{p}{p-2}},$$
where we take $\tilde \sigma =\sigma^{p-2}$. This is precisely the porous media equation, and its explicit solutions can be found by separation of variables according to \cite[Chapter 4]{MR2286292}.
This observation motivates us to consider $\sigma(t,x)$ of the form $\frac{1}{\sqrt{1-t}}h(x)$. The first order condition of $\sigma(t,M_t)^{p}$ being martingale yields that 
\begin{align*}
0&=\partial_t \frac{h^{p}(x)}{(1-t)^{p/2}}+ \frac{1}{2} \frac{h^2(x)}{1-t}\partial_{xx}^2 \frac{h^{p}(x)}{(1-t)^{p/2}} \\
&=\frac{p h^p(x)}{2(1-t)^{(p+2)/2}}+ \frac{h^2(x)\partial_{xx}^2 h^p(x)}{2(1-t)^{(p+2)/2}},
\end{align*}
which implies that $0=ph^{p-2}(x)+ \partial^2_{xx}  h^p(x).$
Denoting $y(x):=h^{p}(x)$, we get an autonomous ODE 
\begin{align}\label{eq:autODE}
0= y''(x)+py^{\frac{p-2}{p}}(x).
\end{align}
Solving \eqref{eq:autODE} with boundary conditions $y(0)=y(1)=0$, we obtain the following result, {whose proof can be skipped on a first reading. }
\begin{proposition}\label{prop:imsoln}
Fix $p \in (0,\infty)$. With the boundary condition $\sigma(t,0)=\sigma(t,1)=0$, $t \in [0,1]$, \eqref{eq:pmd} has a nonnegative solution such that for $(t,x) \in [0,1)$$ \times [0,1/2]$
\begin{align*}
\partial_x \sigma(t,x)=-\partial_x \sigma(t,1-x)= 
\begin{cases}
\frac{1}{p \sqrt{1-t}} \sqrt{\frac{p^2}{1-p}-C_p (1-t)^{1-p}\sigma^{2-2p}(t,x)}, \quad &\text{if } p \in (0,1), \\
\frac{1}{p \sqrt{1-t}} \sqrt{C_p (1-t)^{1-p}\sigma^{2-2p}(t,x)-\frac{p^2}{p-1}}, \quad &\text{if } p \in (1,\infty), \\
\frac{1}{\sqrt{1-t}} \sqrt{-2 \log \sigma(t,x)-\log(1-t)-C_1}, \quad &\text{if } p=1,
\end{cases}
\end{align*}
where $C_p$ is a unique positive constant (in particular $C_{1/2}=\sqrt{2}$, $C_1=\log(2\pi)$). Furthermore we have that $ \sigma(t,x) \leq \frac{1}{\sqrt{1-t}}\left(\frac{|2p-2|C_p}{2p^2} \right)^{\frac{p}{2p-2}}$ if $p \not =1$, that $ \sigma(t,x) \leq \frac{1}{\sqrt{1-t}}e^{-C_1/2}$ if $ p=1$, and that $ \sigma(t,x)=0$ only at $x=0,1$. 
\end{proposition}
\begin{proof}
It is sufficient to solve \eqref{eq:autODE}. Multiplying \eqref{eq:autODE} by $2\frac{d y}{d x}$ and integrating w.r.t. $x$, in the case that $p \not =1$ we obtain a new equation $$\left(\frac{d y}{d x} \right)^2=- 2 \int y^{\frac{p-2}{p}} \,dy \pm C=  \frac{2p^2}{2-2p} y^{\frac{2p-2}{p}} \pm C, $$
Thanks to the boundary condition $y(0)=y(1)=0$, we could guess that $y(x)=y(1-x)$ for all $x \in (0,1)$ and hence $\frac{dy}{dx}\big|_{x=1/2}=0$.

In the case that $0<p<1$, $\frac{2p^2}{2-2p}>0$ and therefore $\frac{dy}{dx}=0$ at $y_0=\left(\frac{(2-2p)C}{2p^2} \right)^{\frac{p}{2p-2}}$. Noting that $\frac{dx}{dy}=\frac{1}{\sqrt{\frac{2p^2}{2-2p} y^{(2p-2)/p} - C}}$, we choose $C$ so that 
\begin{align*}
\frac{1}{2}&=\int_0^{y_0} \frac{1}{\sqrt{\frac{2p^2}{2-2p} y^{(2p-2)/p} - C}} \,dy =\frac{y_0}{\sqrt{C}}\int_0^1 \frac{1}{\sqrt{z^{(2p-2)/p} - 1}} \,dz\\
&=\left(\frac{2-2p}{2p^2} \right)^{\frac{p}{2p-2}} C^{\frac{1}{2p-2}}\int_0^1 \frac{1}{\sqrt{z^{(2p-2)/p} - 1}} \,dz,
\end{align*}
where we change the variable $z=y/y_0$. Since $\int_0^1 \frac{1}{\sqrt{z^{(2p-2)/p} - 1}} \,dz$ is finite, there exists a unique $C_p>0$ so that the above equality holds, {and in the case of $p=1/2$ one can easily get $C_{1/2}=\sqrt{2}$}. Therefore, we obtain that 
\begin{align*}
x= \int_0^y \frac{1}{\sqrt{\frac{2p^2}{2-2p} \lambda^{(2p-2)/p} - C_p}} \,d\lambda, \quad y \in [0,y_0], 
\end{align*}
which implicitly provides a solution to \eqref{eq:autODE} over $x \in [0,1/2]$, and we can extend the solution symmetrically to $[0,1]$. 

In the case that  $p>1$, $\frac{2p^2}{2-2p}<0$, and by a similar argument the solution is implicitly given by 
\begin{align*}
x=\int_0^y \frac{1}{\sqrt{C_p-\frac{2p^2}{2p-2} \lambda^{(2p-2)/p} }} \,d\lambda, \quad y \in [0,y_0],
\end{align*} 
where $y_0=\left(\frac{(2p-2)C_p}{2p^2} \right)^{\frac{p}{2p-2}}$ and $C_p$ is the unique positive solution of 
\begin{align*}
\frac{1}{2}= \left(\frac{2p-2}{2p^2} \right)^{\frac{p}{2p-2}} C^{\frac{1}{2p-2}}\int_0^1 \frac{1}{\sqrt{1-z^{(2p-2)/p} }} \,dz. 
\end{align*}

If $p=1$ we have instead $$ \left(\frac{d y}{d x} \right)^2=- 2 \int y^{-1} \,dy - C=-2 \log y -C.$$
Solving $\log y =-C/2$, we get $y_0=e^{-C/2}$, and therefore 
\begin{align*}
x=\int_0^y \frac{1}{\sqrt{-2 \log \lambda -C_1}} \, d \lambda, \quad y \in [0,y_0],
\end{align*}
where {$C_1=\log(2\pi)$} is the unique positive solution of 
\begin{align*}
\frac{1}{2}=\int_0^{e^{-C/2}} \frac{1}{\sqrt{-2\log y -C}} \, dy=\frac{e^{-C/2}}{\sqrt{2}}\int_0^{1} \frac{1}{\sqrt{-\log y }} \, dy.
\end{align*}

In the end, noticing that $\sigma(t,x)=\frac{1}{\sqrt{1-t}}y^{1/p}(x)$ and $\partial_x \sigma(t,x)= \frac{1}{p\sqrt{1-t}}y(x)^{(1-p)/p}\partial_x y(x)$, we obtain the results by direct computation. 

\end{proof}

So for every $p>0$, we have a candidate win martingale
\begin{align}\label{eq:optimum}
d\bar{M}^{s,x}_t&=\bar \sigma(t,\bar{M}^{s,x}_t) \, dB_t,\\
\bar M_s&=x, \notag
\end{align}
where $\bar\sigma$ is the solution in Proposition~\ref{prop:imsoln} for the given parameter $p$. { We will prove in the next section that $(\bar{M}_t^{0,x})$ is indeed the optimizer of $\text{OPT}(p,x)$. The reader may assume that $s=0$  in \eqref{eq:optimum} for the moment, but the general notation will be used shortly in Section~\ref{sec:verification}.}  Applying \cite[Theorem 5.5.7]{MR1121940} to the time-scaled martingale $\bar{M}^{0,x}_{1-e^{-r}}$ with $r \in [0,\infty)$, the above SDE admits a unique weak solution on $[0,1)$. Observe that, for $y\in\{0,1\}$, if $\bar{M}^{s,x}_\ell = y$ then also $\bar{M}^{s,x}_t = y$ for all $t\in(\ell,1)$ since $\bar \sigma(t,0)=\bar \sigma(t,1)=0$. In particular then we have $0\leq\inf_{t\in[s,1)}\mo_t^{s,x}\leq \sup_{t\in[s,1)}\mo_t^{s,x}\leq 1$ a.s. Hence the martingale is bounded in $L^p$ for every $p$ and in particular $\mo^{x,s}_1:= \lim_{t\to 1} \mo^{x,s}_t$ exists a.s.\ and in $L^2$. Thus $\mo^{x,s}_1\in [0,1]$ and  $\mathbb E[\langle \mo^{x,s}\rangle_1]<\infty$, hence also 
$$\mathbb E\left[\int_s^1 \bar \sigma^2(t,\bar{M}_t^{s,x})\, dt \right ]<\infty,$$
and in particular $\int_s^1 \bar \sigma^2(t,\bar{M}_t^{s,x})\, dt <\infty$ a.s. We conclude that the event $\{\mo^{s,x}_1\in (0,1)\}$ is negligible since on this event $ \int_s^1 \bar\sigma^2(t,\bar{M}_t^{s,x})\, dt=+\infty$. 

Let us also take  $L_t:=\bar \sigma^p(t, \bar M^{s,x}_t)$. According to \eqref{eq:pmd}, $L_t$ is a local martingale. Due to Proposition~\ref{prop:imsoln}, $L_t$ is uniformly bounded over $[0,1-\varepsilon)$ and hence is a true martingale for any $\varepsilon >0$. 

We summarize the discussion above:

\begin{lemma}\label{lem:sigma_is_mart}
	$\mo^{s,x}$ is well-defined on the whole interval $[s,1]$, it is a continuous martingale bounded in every $L^p$, and it satisfies $\mo^{s,x}_1\in\{0,1\}$ a.s. (implying that $\mo^{s,x}_1\sim Bernoulli(x)$). Furthermore, the process $L_t:=\bar \sigma^p(t, \bar M^{s,x}_t)$ is also a martingale on $[0,1)$.
\end{lemma}

\begin{remark}
For $p \in (0,1)$, given any $\varepsilon >0$, thanks to Proposition~\ref{prop:imsoln}, $x \mapsto \bar \sigma(t,x)$ is uniformly Lipschitz for $t  \in [0,1-\varepsilon)$. Therefore, we have a strong solution to \eqref{eq:optimum}. 
\end{remark}

{ 
\begin{remark}\label{rmk:wf-diffusion}
  From Proposition~\ref{prop:imsoln}, in the case of $p=2$ we have $\bar\sigma(t,x)=\sqrt{\frac{x(1-x)}{1-t}}$, and the corresponding martingale $d\bar M_s= \sqrt{\frac{\bar M_s(1-\bar M_s)}{1-t} }\, dB_t$. After time change,  $Y_t:=\bar M_{1-e^{-t}}$ satisfies $dY_t=\sqrt{Y_t(1-Y_t)} \, dB_t$, which is the so-called Wright-Fisher martingale diffusion; see e.g.\ \cite[Chapter 15]{KaTa81}. Thus this win-martingale is singled out, despite the fact that for $p=2$ all win-martingales are optimal.
\end{remark}
}

Let us discuss some other explicit solutions, and how these compare to each other. As we discuss in detail in Section \ref{sec:intriguing}, one can verify that $\bar{\sigma}(t,x)= \sqrt{\frac{2}{1-t}}x(1-x)$ satisfies \eqref{eq:pmd} and Proposition~\ref{prop:imsoln} with $p=\frac{1}{2}$, which gives rise to the SDE
\[d \bar{M}_t= \sqrt{\frac{2}{1-t}}\bar{M}_t (1-\bar{M}_t) \, dB_t. \] 
In the case of $p=1$, the volatility function $\tilde{\sigma}$ given by Proposition~\ref{prop:imsoln} yields a win martingale $\tilde{M}$, which is a particular case of a so-called Bass martingale \cite{backhoff2020martingale,2023arXiv230611019B}. Indeed, these authors consider the problem of maximizing $\mathbb E[\int_0^1\sigma_t \,dt]$ over martingales satisfying initial and terminal distributional constraints. Explicitly, $\tilde{M}_t = \Phi_{1-t}(B_t)$ with $\Phi_{1-t}$ the cdf of the centered Gaussian with variance $1-t$.

\section{Verification of optimality}\label{sec:verification}

In this section, we verify that the candidate martingale $(\bar{M}^{0,x}_t)_{t \in [0,1]}$ is the optimizer for $\text{OPT}(p,x)$ in \eqref{eq:inf_divergence_prob} (maximizer for $p \in (0,2)$ and minimizer for $p>2$).  Associated to the martingale $\mo$ we define its cost
\begin{align}\label{eq:def_v}
\vo(s,x):=\mathbb E\left[\int_s^1 \bar \sigma^p(t, \bar{M}_t^{s,x}) \,dt  \right ]=(1-s)\bar \sigma^p(s,x),
\end{align}
where the second equality is due to Lemma~\ref{lem:sigma_is_mart} and Fubini's theorem. 

\begin{lemma}\label{lem:second_var}
For $p\in (0,2)$ ($p \in (2,\infty)$ $respectively$), $\bar \sigma(t,x)$ is the unique maximizer (minimizer $respectively$) of the function
$$[0,\infty)\ni\sigma \mapsto \frac{1}{2}\sigma^2 \partial^2_{xx}\vo(t,x)+\sigma^p.$$
\end{lemma}

\begin{proof} 

We only prove the result for the case $p \in (0,2)$, and the argument for $p>2$ is similar. Due to \eqref{eq:autODE} and $\bar \sigma^p(t,x)= \frac{1}{(1-t)^{p/2}}y(x)$, it can be seen that 
\begin{align*}
 \partial_{xx}^2 \vo(t,x)= \frac{(1-t)}{ (1-t)^{p/2}}  y''(x)= -\frac{p}{ (1-t)^{(p-2)/2}} y^{\frac{p-2}{p}}(x) <0.
\end{align*}
As $F(\sigma):= \frac{1}{2}\sigma^2 \partial^2_{xx}\vo(t,x)+\sigma^p$ is regular, local maximums are obtained at either boundaries $0,\infty$ or stationary points. Then the first order condition yields that 
\begin{align*}
0=\sigma \partial_{xx}^2 \vo(t,x)+p \sigma^{p-1}.
\end{align*}
Solving the equality above, we get the stationary point $\sigma= \left(\frac{-\partial_{xx}^2 \vo(t,x)}{p}\right)^{1/(p-2)}=\frac{1}{\sqrt{1-t}}y^{1/p}(x)=\bar \sigma(t,x).$ Noticing that $F(\bar \sigma(t,x))=(1-p/2) (1-t)^{-p/2}y(x)>0$, $\lim\limits_{\sigma \to 0}F(\sigma)=0$, $\lim\limits_{\sigma \to \infty} F(\sigma)=-\infty$, $F(\sigma)$ obtains its unique maximizer at $\sigma=\bar\sigma(t,x)$.

\end{proof}

With the result above, we can now verify that the function $\vo$ satisfies the HJB equation of optimization problem \eqref{eq:inf_divergence_prob} strictly before time $1$. 

\begin{lemma}\label{lem_HJB}
	On $[0,1)\times[0,1]$, we have that for $p \in (0,2)$,
\begin{align*}
\partial_t \vo(t,x) +\sup_{\sigma\geq 0}\, \left\{\frac{1}{2}\sigma^2\, \partial^2_{xx}\vo(t,x)+\sigma^p\right\}=0, 
\end{align*}
and in the case that $p \in (2,\infty)$,
\begin{align*}
\partial_t \vo(t,x) +\inf_{\sigma\geq 0}\, \left\{\frac{1}{2}\sigma^2\, \partial^2_{xx}\vo(t,x)+\sigma^p\right\}=0.
\end{align*}
\end{lemma}

\begin{proof}
	By \eqref{eq:def_v} and the Markovian property of $\bar M$, we have that
	\[t\mapsto \vo(t,\mo_t^{0,x})+\int_0^t \bar \sigma^p(u, \bar{M}_u^{0,x}) \, du \]
	is a martingale. Thanks to It\^{o}'s formula, this means that
	\[\partial_t \vo(t,z) +\frac{1}{2}\so^2(t,z)\partial_{zz}\vo(t,z) + \bar\sigma^p(t,z)=0.\]
	But then by  Lemma \ref{lem:second_var} the l.h.s.\ above is equal to $\partial_t \vo(t,x) +\sup_{\sigma\geq 0}\, \left\{\frac{1}{2}\sigma^2\, \partial^2_{xx}\vo(t,x)+\sigma^p\right\}$ when $p \in (0,2)$, and equal to $\partial_t \vo(t,x) +\inf_{\sigma\geq 0}\, \left\{\frac{1}{2}\sigma^2\, \partial^2_{xx}\vo(t,x)+\sigma^p\right\}$ when $p \in (2,\infty)$.
\end{proof}

Before implementing the verification argument, we would like to mention that the proof for $p>2$ is subtler than for $p<2$. 

Let us introduce the value function of the minimization problem 
\[ 
v(t,x):=\inf\left\{ \E\left[ \int_t^1 \Sigma_u^{p/2} \, du \right]: \, \mathbb{Q} \in \mathcal{M}^c_{t,x,win}  \right\},
\]
where $\mathcal{M}^c_{t,x,win}$ denotes the set of distributions of continuous win-martingales over time $[t,1]$ that starts with $x$ at time $t$. 
Similarly as in Example~\ref{ex:infty}, by Jensen's inequality we have $v(t,x) \to \infty $ as $t \to 1$ for $x \in (0,1)$. Therefore the terminal condition of value function $v(t,x)$ at $t=1$ is irregular. It implies that the natural terminal condition for its HJB equation
\begin{align*}
0 = \partial_t v(t,x) + \inf_{\sigma \geq 0}\left\{\frac{1}{2}\sigma^2 \partial_{xx}^2 v(t,x)+\sigma^p \right \}
\end{align*}
is given by 
\begin{align*}
\begin{cases}
\, \, 0=v(1,x) \quad \text{$x \in \{0,1\}$},     \\
\infty=v(1,x)  \quad  \text{$x \in (0,1)$}.
\end{cases}
\end{align*}
Although it is degenerate parabolic, little is known, to the authors' knowledge, due to the irregular boundary condition. 

To carry out the verification argument, we want to show that for any feasible martingale $(M_t)_{t \in [0,1]}$ with volatility $(\sigma_t)_{t \in [0,1]}$, the process $\vo(t,M_t)+\int_0^t \sigma_s^p \, ds $ is a sub-martingale.  Since $\partial_{x} \vo(t,x)$ is uniformly bounded for $t \in [0,1-\epsilon]$, it is a sub-martingale before time $1$ as shown in Lemma~\ref{lem_veri}. So it reduces to the question whether $\mathbb{E}[\vo(t, M_t)] \to 0$ as $t \to 1$ for all admissible martingales $M$. The answer is affirmative due to the estimate in Lemma~\ref{lem:uniformbound}.

In the rest of this section, we first provide two technical lemmas for the case $p>2$, and then prove the main result in both cases.

\begin{lemma}\label{lem_veri}
Fix $p>2$.	Let $M$ be feasible for our minimization problem (started from $x_0$ at time 0), and denote by $\sigma_t$ the  square root of the density of its quadratic variation. Then the process
	$$t\mapsto R_t^M:=\vo(t,M_t)+\int_0^t \sigma_s^{p} \, ds
$$
is a submartingale on $[0,1)$.
\end{lemma}

\begin{proof}
	By Lemma \ref{lem_HJB} we have
	\[\partial_t \vo(t,M_t)+  \frac{1}{2}\sigma_t^2 \partial^2_{xx}\vo(t,M_t)+\sigma_t^p \geq 0,\]
	from which, thanks to It\^{o} formula, the local submartingale property of $R^M$ follows. The local martingale part of $R_t^M$ is given by the stochastic integration 
	\begin{align*}
	\int_0^t \partial_x \vo(s, M_s) \sigma_s \, d B_s.
	\end{align*}
Thanks to Proposition~\ref{prop:imsoln}, $\partial_x \vo(s,M_s)$ is uniformly bounded over $[0,t]$ for $t <1$ and hence $\partial_x \vo(s, M_s) \sigma_s$ is square-integrable over $[0,t]$. Therefore the stochastic integral is indeed a martingale and it concludes the result. 
\end{proof}

\begin{lemma}\label{lem:uniformbound}
Fix $p>2$. Let us introduce $$\tilde v(t,x):=(1-t)^{1-p/2}\left((1-x)^p x+ (1-x) x^p \right).$$ Then there exist two positive constants $c_1,c_2$ such that 
\begin{align}\label{eq:uniformbound}
  c_1 \tilde v(t,x) \leq  v(t,x) \leq \vo(t,x) \leq c_2 \tilde v(t,x), \quad \forall \, (t,x) \in [0,1) \times [0,1]. 
\end{align} 
\end{lemma}
\begin{proof}
Given any feasible martingale $M$ starting from $x$ at time $t$, we have by Jensen's inequality that 
\begin{align*}
\E\left[\int_t^1 \sigma_s^p \, ds \right]\geq (1-t)^{1-p/2} \E \left[ \left(\int_t^1 \sigma_s^2 \,ds  \right)^{p/2}\right]=(1-t)^{1-p/2}\E \left[(\langle M \rangle_1-\langle M \rangle_t)^{p/2}\right].
\end{align*}
Thanks to Doob's martingale inequality and BDG inequality, the last term on the right is bounded from below by $$c_1(1-t)^{1-p/2}\E [(M_1-M_t)^{p}]=c_1(1-t)^{1-p/2}\left((1-x)x^p+x (1-x)^p \right),$$
where $c_1$ is some positive constant. Therefore due to the definition of value function, we obtain the first inequality of \eqref{eq:uniformbound}.

It remains to show that  $$\sup_{(t,x) \in [0,1) \times (0,1) }\frac{\vo(t,x)}{\tilde v(t,x)} <+\infty . $$
Thanks to Proposition~\ref{prop:imsoln} and \eqref{eq:def_v}, $\vo(t,x)=(1-t)^{1-p/2}y(x)$ , where $y(x)$ is a solution to \eqref{eq:autODE} that given implicitly in Proposition~\ref{prop:imsoln}. According to L' Hospital rule, 
\begin{align*}
\lim\limits_{x \to 0} \frac{\vo(t,x)}{\tilde v(t,x)}= \lim\limits_{x \to 0}\frac{y'(x)}{\left((1-x)^p x+ (1-x) x^p \right)'}=\sqrt{C_p}.
\end{align*} 
By the same token, $\lim\limits_{x \to 1} \frac{\vo(t,x)}{\tilde v(t,x)}=\sqrt{C_p}$, and hence $\frac{\vo(t,x)}{\tilde v(t,x)}$ is uniformly bounded over $[0,1]$. 
\end{proof}

We can now carry on the verification argument, showing the optimality of $\mo$:

\begin{theorem}\label{thm:veri}
For $p \in (0,2) \cup(2,\infty)$, the unique optimizer of \eqref{eq:inf_divergence_prob} is $\bar M$.
\end{theorem}
\begin{proof}
\emph{Step 1: $p \in (0,2)$.} Given any feasible martingale $M$ (started from $x_0$ at time $0$), denote by $\sigma_t$ the  square root of the density of its quadratic variation. Due to Lemma~\ref{lem_HJB},
\[\partial_t \vo(t,M_t)+  \frac{1}{2}\sigma_t^2 \partial^2_{xx}\vo(t,M_t)+\sigma_t^p \leq 0,\]
and therefore $t \mapsto \vo(t,M_t)+\int_0^t \sigma^p_s \,ds$ is a local super-martingale. Since it is also nonnegative, thanks to Fatou's lemma is a true super-martingale. Therefore we conclude that 
\begin{align*}
	\mathbb E\left [
\int_0^1 \sigma_t^p \,ds	\right ] &=
	\mathbb E\left [
\vo(1,M_1)+\int_0^1 \sigma_t^p \, ds	\right ]
\leq 	\vo(0,x_0). 
\end{align*}

\vspace{4pt}
\emph{Step 2: $p \in (2,\infty)$.} Let $M$ be any feasible martingale that starts from $x_0$  at time $0$ and $\E[ \int_0^1 \sigma_t^p \, dt]<+\infty$. Invoking Lemma~\ref{lem:uniformbound}, we get that 
\begin{align*}
\E\left[\int_0^1 \sigma_s^p \, ds \right] &\geq \E \left[\int_0^t \sigma_s^p \, ds \right]+ \E [v(t, M_t)] \\
&\geq \E \left[\int_0^t \sigma_s^p \, ds \right]+ c_2 \E[ \tilde v(t,M_t)],
\end{align*}
which indicates that $\lim\limits_{t \to 1} \E[\vo(t,M_t)]=\lim\limits_{t \to 1} \E[\tilde v(t,M_t)]=0$. According to Lemma	 \ref{lem_veri}, for any $t <1$ we have
	\begin{align*}
	\mathbb E\left [\vo(t,M_t)+
\int_0^t \sigma_s^p \,ds	\right ] &
\geq 	\vo(0,x_0)
= \mathbb E\left [
\int_0^1 \so^p(s, \mo^{0,x_0}_s) \, ds	\right ],
	\end{align*}
and hence we conclude the result by letting $t \to 1$ . 

\vspace{4pt}

\emph{Step 3: Uniqueness.} We close this theorem by showing the uniqueness of optimizers to Problem \eqref{eq:inf_divergence_prob}: As the previous proofs show, the only way for $\sigma$ to be optimal is by making \[\partial_t \vo(t,M_t)+ \frac{1}{2} \{\sigma_t^2 \partial^2_{xx}\vo(t,M_t)+\sigma^p_t \} \]
be equal to zero.  By Lemma \ref{lem:second_var} this is only achieved by $\bar\sigma$.

\end{proof}

\begin{remark}\label{rmk:earlytermination}
In the case of $p \in (0,2)$, we observe that the optimal win martingale $(\mo^{t,x}_s)_{s \in [t,1]}$ is also the unique solution of the maximization problem
\begin{align*}
w(t,x):=\sup\left\{\E^{\Q}\left[\int_t^{\tau \wedge 1}\Sigma_u^{p/2} \,du \right]: \, \Q \in \mathcal{M}^c_{t,x}, \, \tau:= \inf\{u \geq t: X_u \not \in (0,1)\} \right\},
\end{align*}
where $\mathcal{M}^c_{t,x}$ denotes the set of laws of continuous martingale over time $[t,1]$  which have absolutely continuous quadratic variation and start at $x$ at time $t$. 
Compared with \eqref{eq:inf_divergence_prob}, this problem relaxes the constraint of the terminal distribution being  $Bernoulli(x)$, by allowing early termination before time $1$ in case the martingale tries to leave the interval $[0,1]$, and otherwise permitting an arbitrary distribution at time $1$.

Indeed, by a standard argument, it can be verified that the value function $w: [0,1] \times [0,1] \to \mathbb{R}_+$ is the unique bounded viscosity solution of its corresponding HJB equation 
\begin{align*}
0&= \partial_t w(t,x)+\sup_{\sigma \geq 0}\left\{ \frac{1}{2}\sigma^2 \partial_{xx}^2 w(t,x)+\sigma^p \right\},  \\
0&=w(t,0)=w(t,1)=w(1,x), \quad (t,x) \in [0,1] \times [0,1],
\end{align*}
which is solved by the value of the optimal win martingale $\vo: [0,1] \times [0,1] \to \mathbb{R}_+$. Therefore the same argument as in Theorem~\ref{thm:veri} completes our claim. 

Actually, this maximization problem is an analogue of \cite{2023arXiv230607133G}, i.e., we take $\sigma_t^{p} $ as the objective function instead of the cost $ \log (\sigma_t^2)+1$ considered in \cite{2023arXiv230607133G}. Moreover, given the logarithm as objective function, results of \cite{2023arXiv230607133G} imply that allowing early termination $\tau$ makes the optimization problem different. Precisely, the maximizer of
\begin{align*}
\sup\left\{\E^{\Q}\left[\int_0^{\tau \wedge 1} \log(\sigma_t^2)+1 \,dt \right]: \, \Q \in \mathcal{M}^c_{x,win}, \, \tau:= \inf\{s \geq t: X_s \not \in (0,1)\} \right\}
\end{align*}
is different from that of 
\begin{align*}
\sup\left\{\E^{\Q}\left[\int_0^{ 1} \log(\sigma_t^2)+1 \,dt \right]: \, \Q \in \mathcal{M}^c_{x,win}.\right\}
\end{align*}
In contrast, as justified above, when the objective function is $\sigma_t^{p}$, our win martingale $(\mo^{t,x}_s)_{s \in [t,1]}$ is optimal no matter if possible early termination is allowed or not. 

\end{remark}

The verification argument of this section follows the lines of the corresponding argument in \cite{BBexciting}. We stress here some differences: Due to the different choice of cost functionals, the candidate value function $\bar u$  in \cite{BBexciting} tends to $\infty$ near the boundary $x=0,1$ and in our framework ($p>2$) $\vo$ explodes near terminal time $t=1$. Therefore in order to finish the verification argument, \cite{BBexciting} estimated $\E[\bar u(\tau_{\epsilon},M_{\tau_{\epsilon}})]$ with $\tau_{\epsilon}=\inf\{ t \geq 0: \, M_t \not \in (\epsilon,1-\epsilon)\}$ uniformly for all admissible martingales $M$, while we make use of a uniform estimate of $\E[\vo(t,M)]$ for $t \to 1$.

{

\subsection{On convex ordering}

Having studied the value function of our problems, we would like to justify a claim from the introduction: At each moment of time $t \in (0,1)$, the distribution of the Bass martingale $\tilde{M}_t$, corresponding to $\bar{M}_t$ for $p=1$ (see \cite{backhoff2020martingale,2023arXiv230611019B}), is more spread out in space than the distribution of $\bar{M}_t$ for $p=1/2$. Actually, we can say more. Recalling the Aldous martingale $\hat{M}$, defined in \cite{BBexciting} through the SDE 
\[d\hat{M}_t=\hat{\sigma}(t,\hat M_t) \,dB_t=\frac{\sin(\pi\hat{M}_t)}{\pi\sqrt{1-t}} \,dB_t, \]
we prove that 
\begin{align}\label{eq:conv_ord}
	\text{Law}(\hat{M}_t) <\text{Law}({M}^{p_1}_t) < \text{Law}(M^{p_2}_t) \,\,\text{in the convex order},
\end{align} 
where $0<p_1<p_2$, $M^{p_1}$, $M^{p_2}$ are solutions to \eqref{eq:optimum} with corresponding parameters $p_1,p_2$ respectively, and $\hat M$, $M^{p_1}$, $M^{p_2}$ start with the same position at time $0$.  The Aldous martingale is characterized by the fact that $\partial_t (\log\hat\sigma) +\frac{1}{2}\hat{\sigma}^2\Delta (\log\hat\sigma) =0 $, and this can be obtained as the formal limit of \eqref{eq:pmd} as $p\to 0$. Hence it can be considered as a limiting optimal win-martingale in our context.  

Denote by $\sigma_{p_1}$, $\sigma_{p_2}$ the volatility of $M^{p_1}$ and $M^{p_2}$ respectively. According to \cite[Theorem 2.1]{hobson1998volatility}, a sufficient condition for \eqref{eq:conv_ord} is $\hat \sigma \leq \sigma_{p_1} \leq \sigma_{p_2}$ pointwise. In the following result, we prove this using the first order condition of being the optimal win-martingale; see Lemma~\ref{eq:FirstOrderMOT}.

\begin{lemma}\label{lem:cx}
    For $(t,x) \in [0,1) \times (0,1)$ and $0<p_1<p_2$, it holds that \[\hat{\sigma}(t,x) < \sigma_{p_1}(t,x) < \sigma_{p_2}(t,x),\]  
    which yields \eqref{eq:conv_ord} thanks to \cite[Theorem 2.1]{hobson1998volatility}. 
\end{lemma}
\begin{proof}
Let us recall the value functions
of the optimization problems 
\begin{align*}
\hat v(t,x):= &\sup\left\{ \E^{\mathbb Q}\left[ \int_t^1 \log \Sigma_u \, du \right]: \, \mathbb{Q} \in \mathcal{M}^c_{t,x,win}  \right\}, \\
v_p(t,x):= &\sup\left\{ \E^{\mathbb Q}\left[ \int_t^1 \Sigma_u^{p/2} \, du \right]: \, \mathbb{Q} \in \mathcal{M}^c_{t,x,win}  \right\}, \quad \quad p \in (0,2], \\
v_p(t,x):= &\inf\left\{ \E^{\mathbb Q}\left[ \int_t^1 \Sigma_u^{p/2} \, du \right]: \, \mathbb{Q} \in \mathcal{M}^c_{t,x,win}  \right\}, \quad \quad  p \in [2,+\infty),
\end{align*}
where $\mathcal{M}^c_{t,x,win}$ denotes the set of distributions of continuous win-martingales over time $[t,1]$ that starts with $x$ at time $t$. Note that in the case of $p=2$, $\E^{\mathbb Q}\left[ \int_t^1 \Sigma_u \, du \right]=x(1-x)$ is independent of the choice of $\mathbb Q \in \mathcal{M}^c_{t,x,win}$, and thus $v_2(t,x)=x(1-x)=(1-t)\sigma_2^2(t,x)$ thanks to Remark~\ref{rmk:wf-diffusion}. 

Suppose $p \in (0,2)$. Due to \eqref{eq:def_v} and Theorem~\ref{thm:veri}, $v_p(t,x)=(1-t) \sigma_p^p(t,x)$. It was shown in \cite{BBexciting} that the Aldous martingale is the maximizer of $\hat v(t,x)$ and $\hat v(t,x)=(1-t)\log \hat\sigma(t,x)^2$. As $\mathbb R \ni x \mapsto \exp(px/2)$ is convex, by Jensen's inequality we obtain that 
\begin{align*}
    \exp\left(\frac{p}{2(1-t)}\mathbb E^{\mathbb Q} \left[\int_t^1 \log \Sigma_s \, ds \right] \right) \leq \frac{1}{1-t} \mathbb E^{\mathbb Q} \left[\int_t^1 \Sigma_s^{p/2} \, ds      \right].
\end{align*}
Taking maximum over $\mathbb Q \in \mathcal{M}^c_{t,x,win}$, we obtain that 
\begin{align*}
    \exp\left(\frac{p}{2(1-t)}\hat v(t,x)  \right) < \frac{1}{1-t}v_p(t,x),
\end{align*}
and hence $p \log \hat\sigma(t,x) < p \log \sigma_p(t,x)$ from which it follows that $\hat\sigma(t,x) <\sigma_p(t,x)$.

Suppose either $0<p_1<p_2 \leq 2$ or $2 \leq p_1 <p_2$. By Jensen's inequality 
\begin{align*}
     \left(\frac{1}{1-t} \mathbb E^{\mathbb Q} \left[\int_t^1 \Sigma_s^{p_1/2} \, ds    \right]\right)^{\frac{p_2}{p_1}}  \leq \frac{1}{1-t} \mathbb E^{\mathbb Q} \left[\int_t^1 \Sigma_s^{p_2/2} \, ds   \right].
\end{align*}
Taking either maximum or minimum over $\mathbb Q$ depending on $p_2 \leq 2$ or $2 \leq p_1$ respectively, we get that 
\begin{align*}
    \left(\frac{1}{1-t}v_{p_1}(t,x)\right)^{\frac{p_2}{p_1}}  \leq \frac{1}{1-t} v_{p_2}(t,x),
\end{align*}
and hence $\sigma_{p_1}(t,x) \leq \sigma_{p_2}(t,x)$.

Finally, suppose $p_1 < 2 <p_2$. From the inequalities above, we get that $\sigma_{p_1}(t,x) \leq \sigma_2(t,x) \leq \sigma_{p_2}(t,x)$, and hence conclude the result. 
\end{proof}
}

\section{The intriguing case of $\sqrt[4]{\Sigma}$}\label{sec:intriguing}

In this section, we discuss the intriguing case when $p=\frac{1}{2}$. Recall that $\sigma(t,x)= \sqrt{\frac{2}{1-t}}x(1-x)$ in the case $p=\frac{1}{2}$, and hence according to Theorem~\ref{thm:veri} the SDE
\begin{align}\label{eq:M}
\begin{cases}
dM_t=\sqrt{\frac{2}{1-t}}M_t(1-M_t)\,dB_t, \\
\,\, \, M_0=x_0,
\end{cases}
\end{align}
is the unique maximizer of \[\sup\left\{\E^{\Q}\left[\int_0^1 \Sigma_t^{1/4}\,dt \right]: \Q\in \mathcal M^c_{x_0, \text{win}} \right\}.\]

{
\begin{remark}
Applying Feller's test as in \cite[Lemma 5.2]{BBexciting}, it can be verified that $M_t$ stays in the interior $(0,1)$ for $t <1$, and hits the boundary at terminal time $1$. Together with Remark~\ref{rmk:earlytermination}, it gives a full picture of the maximization problem with possible early termination in the case of $p=1/2$. 
\end{remark}
}

Now we scale time so that instead of $[0,1]$ we work on $[0,\infty)$. Defining $Y_t:=M_{1-e^{-t/2}}$ we find that
\begin{align}\label{eq:Y}
dY_t=Y_t(1-Y_t) \,dW_t,
\end{align}
where $W$ is a Brownian motion on  $[0,\infty)$. In the following, we will give two interpretations of $(Y_t)$ through the lenses of the Schr\"{o}dinger problem and filtering theory.

\subsection{Connection with the Schr\"{o}dinger problem}
Let $G(x):=\log(x/(1-x))$ and define $C_t:=G(Y_t)$ with $Y$ as above. Then according to It\^{o}'s formula, $(C_t)$ solves the SDE
\[dC_t =  \frac{1}{2}\tanh\left(\frac{C_t}{2} \right) \, dt+dW_t ,\]
and the drift process $\frac{1}{2}\tanh\left(\frac{C_t}{2}\right) =\frac{1}{2}(2Y_t-1) $ is a martingale. The latter is the first order condition of the Schr\"{o}dinger problem of entropy minimization w.r.t.\ Wiener measure subject to fixed initial and terminal distributions; see \cite[Proposition 3.2]{Backhoff:2020aa} with the potential function '$W=0$' in the notation therein. We can in fact justify that the law of $C$ is precisely the law of Brownian motion conditioned to \emph{$W_{T}\sim \pm T/2$ as $T\to  \infty$}. This is, to the best of our knowledge, the first time that a natural connection between continuous-time MOT and the Schr\"{o}dinger problem appears. By contrast, the works \cite{henry2019martingale} in continuous-time, and \cite{nutz2023martingale,nutz2024martingale} in discrete-time, also deal with martingale transport and so-called Schr\"odinger bridges, but the connection between the two subjects is forced by design.

\begin{lemma}\label{lem:aux_dens}
	Let $(C_t)_{t\in[0,\infty)}$ be the unique strong solution with $C_0=c$ of the SDE
	\[dC_t = \frac{1}{2}\tanh(C_t/2)\,dt+dW_t,\]
	where $W$ is a Brownian motion started likewise at $c$. Then for $t\in(0,\infty)$ we have
	\[\frac{d\text{Law}(C_t)}{d\text{Law}(W_t)}(z)=\frac{\cosh(z/2)}{\cosh(c/2)}e^{-t/8}, \,\, z\in\mathbb R.\]
\end{lemma}
\begin{proof}
	We denote by $W$ the canonical process and by $\mathbb P$ the probability measure so that $W$ is a Brownian motion started at $c$. Let 
	\[Z_T:= \exp\left\{\int_0^T\frac{1}{2}\tanh(W_s/2) \, dW_s - \frac{1}{2}\int_0^T\frac{1}{4}\tanh^2(W_s/2) \, ds \right \}.\]
	Since $(\log \cosh)'=\tanh$, and $(\tanh)'=\cosh^{-2}$, we also have by Ito's formula
	\begin{align*}
	Z_T&= \exp\left\{\log\cosh(W_T/2)-\log\cosh(W_0/2) - \frac{1}{8} \int_0^T[\tanh^2(W_s/2)+ \cosh^{-2}(W_s/2)] \,ds \right \}\\
	&= \frac{\cosh(W_T/2)}{\cosh(c/2)}e^{-T/8},
	\end{align*}
where we used that $\tanh^2+\cosh^{-2}=1$.
	By Girsanov theorem and the uniqueness of the SDE for $C$, we have that the law of $(C_t)_{t\in[0,T]}$ is equal to the law of $(W_t)_{t\in[0,T]}$ under $\mathbb Q_T$, where $d\mathbb Q_T :=Z_T\,d\mathbb P$.	
\end{proof}

\begin{remark}
	It follows from this lemma that, if $c=0$, then 
	$$\frac{d\text{Law}(C_t)}{dz} = \frac{1}{2\sqrt{2\pi t}}\left \{\exp\left( -\frac{(z-t/2)^2}{2t} \right) + \exp\left( -\frac{(z+t/2)^2}{2t} \right) \right\},$$
	i.e.\ that $C$ has the same marginal laws as the simple mixture of two Brownian motions with drifts $\pm 1/2$. This connection had already been observed in \cite{2023arXiv230318088M}. {We can readily get the explicit density of $Y$ \eqref{eq:Y} (and $M$ \eqref{eq:M}) from this.}
\end{remark}

For simplicity of presentation we assume here that $C_0:=c= 0$, corresponding to  the case $x_0=1/2$. Fixing $t\in (0,\infty)$ and $T>t$, we consider $\mathbb W_{T,\pm T/2}^0$ the law of the Brownian bridge starting at $0$ at time $0$ and finishing at $\pm T/2$ at time $T$ with equal probabilities. On $[0,T-)$ this law is absolutely continuous w.r.t.\ Wiener measure started at $0$, which we have denoted $\mathbb W^0$. In fact, if $Z_t^T$ is the density of $\mathbb W_{T,\pm T/2}^0$ w.r.t.\ $\mathbb W^0$ restricted to $\mathcal F_t$, i.e. $Z_t^T=\frac{d\mathbb W_{T,\pm T/2}^0}{d\mathbb W^0}\big|_{\mathcal{F}_t}$, then $Z_t^T=f^T(t,X_t)$ where after a few calculations we find
\begin{align}\label{eq:desn_schr}
f^T(t,x):= \sqrt{\frac{T}{T-t}}\cdot \exp^{-\frac{x^2}{2(T-t)}}\cdot \exp^{-\frac{T^2}{8}\left \{\frac{1}{T-t}-\frac{1}{T} \right\} }\cdot \cosh\left(\frac{xT}{2(T-t)} \right). 	
\end{align}
Indeed, one can justify that
$$f^T(t,x)=\frac{(e^{-(x-T/2)^2/2(T-t)}+e^{-(x+T/2)^2/2(T-t)})/ \sqrt{2\pi (T-t)}}{(e^{-(T/2)^2/(2T)} + e^{-(T/2)^2/(2T)} )/\sqrt{2\pi T} },$$
from which \eqref{eq:desn_schr} follows. 
Hence, if we send $T\to\infty$, we obtain
\[Z_t^T \to \cosh(X_t/2)\exp^{-t/8}.\]
According to Lemma \ref{lem:aux_dens} this is precisely the density of $C_t$ w.r.t. $X_t(\mathbb W^0)$. In words: \\

\emph{The law of $C_t$ is the limit, as $T\to\infty$, of the law at time $t$ of Brownian motion conditioned to be $\pm T/2$ at time $T$. }\\

In fact more is true. We first compute
\[\partial_x \log(f^T(s,x))=\frac{T}{2(T-s)}\cdot \tanh\left (\frac{xT}{2(T-s)}\right )-\frac{x}{T-s}, \]
and remember that $s\mapsto\partial_x \log(f^T(s,X_s))$ is the drift of $X$ under $\mathbb W_{T,\pm T/2}^0$. Indeed, since $Z_t=f^T(t,X_t)$ is a martingale under $\mathbb{W}^0$, we have $$dZ_t=\partial_x f^T(t,X_t) \,  dX_t=Z_t \, \partial_x\log( f^T(t,X_t)) \, dX_t,$$ and hence $Z_t=\exp\left(M_t -\langle M \rangle_t \right)$ where $M_t:=\int_0^t \partial_x\log f^T(s,X_s)  \, dX_s$. 
Therefore by Girsanov's theorem $X_t - \langle X, M \rangle_t=X_t - \int_0^t \partial_x \log(f^T(s,X_s)) \, ds$ is a Brownian motion under $\mathbb W_{T,\pm T/2}^0$, i.e., the drift of $X$ is given by $s\mapsto\partial_x \log(f^T(s,X_s))$. We notice that this drift converges to $s\mapsto \frac{1}{2}\tanh(X_s/2)$ as $T \to \infty$, which is precisely the drift of $X$ under the law of $C$. In this sense we can say, as already mentioned in the introduction, that \\

\emph{On every fixed interval $[0,t]$, the law of $\mathbb W_{T,\pm T/2}^0$ restricted to $[0,t]$ converges as $T\to\infty$ to the law of $C$ restricted to $[0,t]$.}
\\

More precisely, we prove the following:
 
\begin{theorem}\label{prop:bridges}
Fix $t>0$. Let us denote the law of $C$ and $\mathbb W_{T,\pm T/2}^0$ restricted to $[0,t]$ by $\mathbb Q$ and $\mathbb P^T$ respectively. Then we have $H(\mathbb Q || \mathbb P^T) \to 0$ as $T \to \infty$. As a corollary, $\mathbb P^T$ converges to $ \mathbb Q$ in total variation. 
\end{theorem}
\begin{proof}
According to the discussion above, $\mathbb Q$ and $\mathbb P^T$ are distributions of SDEs 
\[dX_s=\frac{1}{2}\tanh(X_s/2)\,ds + dW^{\mathbb Q}_s \quad \text{and} \quad dX_s= \partial_x \log(f^T(s,X_s)) \, ds +dW^{\mathbb P^T}_s,\]
where $W^{\mathbb Q}$, $W^{\mathbb P^T}$ are Brownian motions under $\mathbb Q$ and $\mathbb P^T$ respectively. Denoting $\mathcal{E}(M)=(\exp(M_t-\langle M \rangle_t))_{t \in [0,\infty)}$ for any $\mathbb P^T$ martingale $M$, then according to Girsanov's theorem, 
\begin{align*}
\frac{d\mathbb Q}{d \mathbb P^T}=\mathcal{E}\left(\int_0^{\cdot} \left( \frac{1}{2} \tanh(X_u/2)- \partial_x \log(f^T(u,X_u)) \right)\,dW_u^{\mathbb P^T}\right),
\end{align*}
and hence 
\begin{align*}
H(\mathbb Q || \mathbb P^T)=&  \frac{1}{2} \E^{\Q} \left[\int_0^t \left( \frac{1}{2} \tanh(X_u/2)- \partial_x \log(f^T(u,X_u)) \right)^2 \, du \right] \\
=& \frac{1}{2} \E^{\Q} \left[ \int_0^t \left(\frac{1}{2}\tanh(X_u/2)-\frac{T}{2(T-u)}\cdot \tanh\left (\frac{X_uT}{2(T-u)}\right )+\frac{X_u}{T-u} \right)^2 \,du \right].
\end{align*}
Noting that $\tanh$ is bounded, $(X_u)$ is square integrable under $\mathbb Q$, so an application of the dominated convergence theorem completes the first claim. The second claim follows by Pinsker's inequality.
\end{proof}

Now we are ready to explain the connection to the classical Sch\"{o}dinger problem. A simple version of the latter is as follows: Given $\nu \in \mathcal{P}(\mathbb{R})$, the Sch\"{o}dinger problem looks for minimizers of $$\inf\left\{H(\mathbb{Q} || \mathbb{W}^0): \, \mathbb{Q}_0=\delta_0, \mathbb{Q}_T=\nu \right\} .$$ 
Recall the tensorization property of relative entropy $H$, i.e., $H(\mathbb{Q} || \mathbb{W}^0)=H(\nu || \mathcal{N}(0,T))+ \int H(\mathbb{Q}_T^x || \mathbb{W}^0_{T,x})\, \nu(dx)$ where $\mathbb{Q}_T^x$ is the disintegration w.r.t.\ the time $T$ marginal and $\mathbb{W}^0_{T,x}$ is the conditioning law of Brownian motion that ends up with $x$ at time $T$. Hence in the case that $H(\nu || \mathcal{N}(0,T))<+\infty$, the minimizer of corresponding Sch\"{o}dinger problem is uniquely given by $\mathbb{Q}(A):= \int \mathbb{W}^0_{T,x}(A) \, \nu(dx)$ for $A \in \mathcal{B}(\Omega)$. Therefore the conditioning of Brownian motion is akin to the classical Schr\"odinger problem, and this points to an intriguing connection between martingale optimal transport and Schr\"odinger problems. Indeed taking $\nu=\frac{1}{2}(\delta_{-T/2}+\delta_{T/2})$ and $\nu^{\epsilon}=\nu\ast \mathcal{N}(0,\epsilon)$,  then  $\mathbb W_{T,\pm T/2}^0$ is the limit of solutions $\mathbb{Q}^{\epsilon}(A):= \int \mathbb{W}^0_{T,x}(A) \, \nu^{\epsilon}(dx)$  as $\epsilon \to 0$.  Another way to stress this is to recall that the drift of $C$, namely $\frac{1}{2}\tanh(C_t/2)=\frac{2Y_t-1}{2}$, is a martingale, and recalling that this would be precisely the first order optimality condition for the optimizer of the Schr\"{o}dinger problem.
\subsection{$Y$ and a filtering problem}
Recall that $Y$ fulfills
$$dY_t=Y_t(1-Y_t)dW_t,$$
for $t\in [0,\infty)$ and $Y_0=x_0\in(0,1)$. In fact $Y$ can be interpreted from the lens of filtering theory as we now explain (but this is well-known):

Let $\mathbb W$ denote Wiener measure, i.e.\ the law of Brownian motion, while $\bar{\mathbb W}$ is the law of Brownian motion with drift 1. Now denote $\mathbb P(du,d\omega)=(1-x_0)\delta_0(du)\mathbb W + x_0\delta_1(du)\bar{\mathbb W}$ on  $\{0,1\}\times C([0,\infty);\mathbb R)$. In other words this is the law of a process that can be either Brownian motion without drift or with drift equal to 1, depending on an independent $Bernoulli(x_0)$ random variable. 
Finally denote by $X_t(u,\omega)=\omega_t$ the canonical process on $\{0,1\}\times C([0,\infty);\mathbb R)$, and $(\mathcal F_t)$ its filtration. Then the stochastic process
$$P_t:= \mathbb P(u=1|\mathcal F_t)$$
satisfies the SDE of $Y$, i.e.\ is equal in law to $Y$. Indeed, Girsanov theorem shows that
\begin{align*}
P_t &=\frac{x_0d \bar{\mathbb W}|_{\mathcal F_t} }{x_0 d\bar{\mathbb W}|_{\mathcal F_t} + (1-x_0)d\mathbb W|	_{\mathcal F_t}} \\
& = \frac{x_0e^{X_t-t/2}}{x_0e^{X_t-t/2} + (1-x_0)}.
\end{align*}
So $P_t= f(X_t-t/2)$ with $f(x)=\frac{x_0e^x}{x_0e^x+(1-x_0)}$. A few computations show that $f'=f-f^2$ and $f''=f-3f^2+2f^3$. Thus
\begin{align*}
	dP_t&=(P_t-P_t^2)(dX_t-dt/2)+\frac{1}{2}(P_t-3P_t^2+2P_t^3)dt \\
	&=(P_t-P_t^2)(P_td_t + dB^X_t-dt/2)+\frac{1}{2}(P_t-3P_t^2+2P_t^3)dt \\
	& = P_t(1-P_t)dB^X_t +\frac{1}{2}(2P_t^2-2P_t^3-P_t+P_t^2+ P_t-3P_t^2+2P_t^3)dt \\
	&= P_t(1-P_t)dB^X_t, 
\end{align*}
where for the second equality we used that $dX_t=P_tdt+dB^X_t$, where $B^X$ is some Brownian motion adapted to $(\mathcal F_t)$ under $\mathbb{P}$. Indeed, this follows from the fact that $\mathbb E[X_t-X_r|\mathcal F_r]=(t-r)P_r=\mathbb E [\int_r^t P_s \,ds |\mathcal F_r]$, valid for $r\leq t$, which shows that $X_t - \int_0^t P_sds$ is a martingale in the aforementioned filtration.

\bibliographystyle{abbrv}
\bibliography{ref}

\end{document}